\documentclass[reqno]{amsart}
\usepackage{relsize}
\usepackage{scalerel}
\usepackage{dsfont}
\usepackage{amsmath, amsthm, amssymb}
\usepackage[usenames,dvipsnames,svgnames,table]{xcolor}
\usepackage[margin=1.5in]{geometry}
\usepackage{enumerate}
\usepackage{dsfont}
\usepackage{comment}
\usepackage{bbm}
\usepackage{mathrsfs}
\usepackage{hyperref}
\usepackage{mathtools}
\mathtoolsset{showonlyrefs}
\allowdisplaybreaks

\def\c{\mathbf{c}}
\def\L{\mathcal{L}}
\def\mch{\mathcal{H}}
\def\mcv{\mathcal{V}}
\def\tdu{\widetilde{U}}
\def\weakconv{\rightharpoonup}          
\newtheorem{theorem}{Theorem}[section]
\newtheorem{proposition}[theorem]{Proposition}
\newtheorem{lemma}[theorem]{Lemma}
\newtheorem{definition}[theorem]{Definition}

\theoremstyle{remark}
\newtheorem{remark}{Remark}

\newcommand{\E}{\mathbb E}
\newcommand{\N}{\mathbb N}

\newcommand{\Pb}{\mathbb{P}}

\newcommand{\Fc}{\mathcal{F}}
\newcommand{\Fct}{\left( \mathcal{F}_t \right)_{t \geq 0}}

\newcommand{\Uc}{\mathscr{U}}
\newcommand{\Xc}{\mathcal{X}}

\newcommand{\Sc}{\mathcal{S}}
\newcommand{\be}{\begin{equation}}
\newcommand{\ee}{\end{equation}}

\numberwithin{equation}{section}

\title[Stochastic Nernst-Planck-Navier-Stokes equations]{Global well-posedness and enhanced dissipation for the 2D stochastic Nernst-Planck-Navier-Stokes equations with transport noise}
\author{Quyuan Lin}
\address{School of Mathematical and Statistical Sciences, Clemson University, Clemson, SC 29634}
\email{quyuanl@clemson.edu}
\author{Rongchang Liu}
\address{Department of Mathematics, University of Arizona, Tucson, AZ 85721}
\email{lrc666@math.arizona.edu}
\author{Weinan Wang}
\address{Department of Mathematics, University of Oklahoma, norman, OK 73019}
\email{ww@ou.edu}
\date{\today}

\begin{document}
\begin{abstract}
    In this paper, we consider the 2D stochastic Nernst-Planck-Navier-Stokes equations with transport noise. By assuming the ionic species having the same diffusivity and opposite valences, we prove the global well-posedness of the system. Furthermore, we illustrate the enhanced dissipation phenomenon in the system with specific transportation noise
    by establishing that it enables an arbitrarily large exponential convergence rate of the solutions. 
\end{abstract}

\maketitle

MSC Subject Classifications: 35Q35, 35Q86, 60H15, 76M35\\

Keywords: Stochastic Nernst-Planck-Navier-Stokes equations, transport noise, global well-posedness, enhanced dissipation

\section{Introduction}

\subsection{The Nernst-Planck-Navier-Stokes system}
Electrodiffusion refers to the process in which ions, possessing different valences, undergo movement characterized by the combination of advection and diffusion. This movement is driven by the interplay of factors such as an electric potential, the individual concentration gradients of ions, and the presence of fluid flow. Simultaneously, electric forces generated by the ions compel the fluid to move. In this paper, we consider electrodiffusion of ions in fluids driven by a transport noise, which takes the effect of small scale randomness into account.

Studies of inoic electrodiffusion have been extensively addressed in different branches of science, bringing forth outstanding applications in the real world. There are tremendous applications of electrodiffusion occurrences in
semiconductors \cite{biler1994debye, gajewski1986basic, mock1983analysis, gao2014high}, in water purification, desalination, and ion separations \cite{alkhad2022electrochemical,zhu2020ion,lee2018diffusiophoretic,yang2019review,gao2014high,lee2016membrane}, and in ion selective membranes \cite{davidson2016dynamical,goldman1989electrodiffusion}.

Mathematically, ionic electrodiffusion is described by the Nernst-Planck (NP) equations, which are given by
\begin{equation} \label{e.w12045}
\partial_t c_i + u \cdot \nabla c_i - D_i \Delta c_i = D_i z_i \nabla \cdot (c_i \nabla \Phi), \quad i=1,...,n,
\end{equation} 
describing the time evolution of the ionic concentrations $c_i$ of $n$ ionic species, with valences $z_i\in\mathbb{R}$ and diffusivities $D_i>0$. Above, the electrical potential $\Phi$ is determined by the Poisson equation
\begin{equation} \label{e.w12046}
- \Delta \Phi = \rho = \sum\limits_{i=1}^{n} z_ic_i.
\end{equation} 
The velocity filed $u$ obeys the two-dimensional incompressible Navier-Stokes system
\begin{equation} \label{e.w12047}
\partial_t u -\nu\Delta u+ u \cdot \nabla u + \nabla p = - \rho \nabla \Phi,\quad \nabla\cdot u=0,
\end{equation} 
where $-\rho\nabla\Phi$ represents the electric force generated by the ions.
The system of equations \eqref{e.w12045}-\eqref{e.w12047} is called the Nernst-Planck-Navier-Stokes (NPNS) system. 

Over the past decade, extensive research has been conducted on the deterministic NPNS system, examining its behavior in the presence and absence of physical boundaries. In the case of two ionic species $c_1$ and $c_2$ with the same diffusivity $D_1=D_2=D$ and opposite valences $z_1=-z_2=1$, which is the simplest setting for the ionic concentrations, the existence, uniqueness, and long-time behavior of solutions were obtained in \cite{ryham2009existence} for $L^2$ large and small initial data on 2D and 3D smooth bounded domains respectively. See also similar results in \cite{schmuck2009analysis} for different boundary conditions. In \cite{constantin2019nernst}, the authors considered the NPNS model with $N$ ionic species with arbitrary valences and diffusivities, which is the most general setting. They achieved global regularity of solutions for $W^{2,p}$ initial data and demonstrated their convergence to stable steady states on 2D bounded smooth domains, incorporating selective boundary conditions. In the periodic setting, the global well-posedness of the NPNS system as well as the exponential stability of solutions were established in \cite{abdo2022space,abdolong}. See also \cite{ALW,IS,ignatova2022global} for global well-posedness results of the Nernst-Planck-Euler and Nernst-Planck-Darcy systems, which are two other closely related models. 

Stochastic effects play a crucial role in the study of a wide range of PDEs. These effects encompass turbulence, random fluctuations, and uncertainty within fluid systems, essential elements for achieving a precise and realistic representation of real-world fluid dynamics in modeling and understanding. There are many results \cite{BT, mikulevicius2004stochastic, ZZ} on the stochastic Navier-Stokes equations. 
Although the deterministic NPNS system has been extensively studied, there is a notable scarcity of research about the stochastic version of NPNS system. In a recent work \cite{abdo2023stochastic}, the authors studied the 2D stochastic NPNS system with equation \eqref{e.w12047} being replaced by the stochastic NSE in the It\^o sense:
\begin{equation}\label{eqn:NSE-sto}
    du + (u \cdot \nabla u  - \nu \Delta u + \nabla p) dt
= - \rho \nabla \Phi dt + f dt+ gdW,
\end{equation}
and they established the global well-posedness of the system, as well as the existence, uniqueness, and smoothness of the ergodic invariant measure. Since $g(x)$ is independent of the solutions, it is referred to as an additive noise. Notably, the equations governing the ionic concentrations do not account for random noise, as the inclusion of additive noise in their evolution would compromise the non-negativity of $c_i$. 

\subsection{NPNS with transport noise and enhanced dissipation} 
In this work, we consider the following stochastic NPNS system driven by a Stratonovich-type transport noise, 
\begin{align}\label{e.w04091}
    \begin{split}
    dc_1&=(D\Delta c_1 -u\cdot \nabla c_1 +D  \nabla \cdot (c_1 \nabla \Phi))dt + b\cdot \nabla c_1 \circ d\mathbb{W}_t
    \\
    dc_2&=(D\Delta c_2 -u\cdot \nabla c_2 -D  \nabla \cdot (c_2 \nabla \Phi))dt + b\cdot \nabla c_2 \circ d\mathbb{W}_t
    \\
    du&=(\nu\Delta u- \Pi(u\cdot \nabla u +\rho \nabla \Phi))dt + \Pi(\theta\cdot \nabla u)  \circ d\mathbb{W}_t , 
    \\
    \nabla\cdot u &=0,
    \\
    -\Delta \Phi &= \rho = c_1 - c_2,
    \end{split}
\end{align}
where $b$ and $\theta$ are divergence-free vector fields, and $\Pi$ denotes the Leray projection. Unlike additive noise, transport noise is employed to model small-scale turbulence and the eddy viscosity phenomenon, which has garnered significant attention over the years, see, for example,  \cite{agresti2022stochastic-2, agresti2022stochastic-1, agresti2021stochastic-2, breit,flandoli2008introduction,LLW,luo2023enhanced, mikulevicius2004stochastic} and references therein. %Based on the literature, the transport noise in the NPNS has the following physical meaning. 
At the microscopic level, the fluid and ionic particles are not only transported by a deterministic drift, but also affected by a space dependent environmental noise, where the latter results in the transport noise at the macroscopic level. Such a viewpoint can be traced back to the work of Reynolds in 1880 as noticed in \cite{mikulevicius2004stochastic} and is closely related to Kraichnan's model of turbulence \cite{kraichnan1968small,kraichnan1994anomalous}. 

Our main results are the well-posedness and enhanced dissipation for the solutions to this NPNS system with Stratonovich-type transport
noise.  
When examining the global well-posedness, the presence of transport noise renders the approach in \cite{abdo2023stochastic} and the entropy estimates in the deterministic case \cite{constantin2019nernst} inapplicable. To address this challenge, we focus on the scenario involving two ionic species with identical diffusivity and opposite valences. This allows us to establish {\it a priori} estimates and extend the local solution to a global one. The non-negativity  of the ionic concentrations plays an important role in the process, which is proved by utilizing a general version of It\^o's formula \cite{kry2013} and a maximal principle type argument. The global well-posedness of the case involving $N$ ionic species with arbitrary valences and diffusivities is left as a future work.

Examining enhanced dissipation in the evolution of a passive scalar $\rho$ (such as temperature distribution or chemical concentration) influenced by fluid motions is a crucial issue in physics and engineering \cite{shraiman2000scalar}. This topic has garnered notable interest from the mathematics community, as evidenced by the recent survey \cite{zelati2023mixing} and the references therein. The evolution of $\rho$ is modeled by the advection-diffusion equation 
\[\partial_t\rho + u\cdot\nabla \rho = \nu\Delta\rho,\]
where $u$ is the ambient flow and $\nu$ is the molecular diffusion. Enhanced dissipation refers to the phenomenon that the presence of a suitable flow $u$ can accelerate the $L^2$ dissipation rate beyond the purely molecular diffusive one. This acceleration is typically due to the shear-diffuse mechanism or the phase mixing of the chaotic flows, see \cite{bedrossian2021almost,zelati2020relation,gess2021stabilization} and references therein.
Recently, enhanced dissipation was also proved in \cite{luo2023enhanced} for the 2D Navier-Stokes equations with a transport noise. By suitably selecting the driving noise, the $L^2$ norm of the velocity field can decay at a rate that can be arbitrarily large. Notably, the enhanced dissipation in \cite{luo2023enhanced} differs from previous studies, as it primarily stems from additional diffusion introduced by the Stratonovich-It\^o corrector, arising from the Stratonovich transport noise modeling small-scale turbulence.

We prove the enhanced dissipation of the stochastic NPNS based on the approach of \cite{luo2023enhanced}, mainly taking the advantage of the Stratonovich transport noise. However, due to the coupled nature of the NPNS system, the initial deviation of the ionic concentrations from their mean values creates a significant electric force on the fluid, leading to turbulence that dominates over the small scale diffusion. Consequently, it is necessary to wait until the concentrations and velocity diminish within a small neighborhood of the equilibrium. This is crucial for the enhanced dissipation generated by the transport noise to come into play, as indicated in Theorem \ref{t.w04093} below. In particular, the size of the neighborhood depends on the viscosity of the fluid and the diffusivity of the concentrations. This is a phenomenon different from \cite{luo2023enhanced}.

To the best of the the authors' knowledge, this is the first work proving global well-posedness and enhanced dissipation of the stochastic NPNS system with transport noise.

\subsection{Organization of the paper} The rest of the paper is organized as follows. In Section \ref{prelim}, we give definitions, introduce basic notations, and state our main results: (i) global well-posedness (Theorem \ref{t.w04092}) and (ii) the enhanced dissipation (Theorem \ref{t.w04093}). In Section \ref{cut-off}, we prove the global existence of martingale solutions and the pathwise uniqueness of solutions to the cut-off system \eqref{sys:modified}. In Section \ref{sec:nonneg}, we establish the non-negativity of ionic concentrations. Building upon this result, we proceed to prove the global well-posedness of the NPNS system \eqref{e.w04091}. In Section \ref{enhance}, we show the enhanced dissipation phenomenon of the stochastic NPNS system.

\section{Preliminaries and main results}\label{prelim}
\subsection{Preliminaries}
The universal constant $C$ appears in the paper may change from line to line. We shall use subscripts to indicate the dependence of $C$ on other parameters, {\it e.g.}, $C_\alpha$ means that the constant $C$ depends only on $\alpha$. We use the notation $a\lesssim b$ to represent $a\leq C b$ for some constant $C$, and $a\lesssim_\alpha b$ to mean that $a\leq C_\alpha b$ with the constant $C$ depends on $\alpha$.

Let $\mathbb{T}^2 = \mathbb{R}^2/2\pi\mathbb{Z}^2$ (so that the constant of the Poincar\'e inequality is $1$) and $\mathbb{Z}_0^2=\mathbb{Z}^2\setminus\{0\}.$  For $1 \leq p \leq \infty$, we denote by $L^p(\mathbb{T}^2)$ the Lebesgue spaces of measurable functions $f$ in $\mathbb T^2$ such that 
\be 
\|f\|_{L^p} = \left(\int_{\mathbb{T}^2} \|f\|^p dx\right)^{1/p} <\infty, \;\text{if } p \in [1, \infty) \quad \text{and} \quad 
%\ee if $p \in [1, \infty)$ and 
%\be 
\|f\|_{L^{\infty}} = \text{esssup}_{x\in\mathbb{T}^2}  |f(x)| < \infty, \; \text{if } p = \infty.
\ee
%\ee if $p = \infty$. 
The $L^2$ inner product is denoted by $\langle\cdot,\cdot\rangle$. For a Banach space $(X, \|\cdot\|_{X})$ and $p\in [1,\infty)$, we consider the Lebesgue spaces $ L^p(0,T; X)$ of functions $f$ in $X$  satisfying 
$
\int_{0}^{T} \|f\|_{X}^p dt  <\infty
$ with the usual convention when $p = \infty$.  For $k \in \mathbb N$, we denote by $H^k(\mathbb T^2)$ the classical Sobolev space of measurable functions $f$ in $\mathbb T^2$ such that $ 
\|f\|_{H^k}^2 = \sum\limits_{|\alpha| \leq k} \|D^{\alpha}f\|_{L^2}^2 < \infty.$
Let $H$ and $V$ be
\be\label{def-zero-mean} 
H = \left\{v\in L^2(\mathbb T^2): \nabla \cdot v = 0,  \, \int_{\mathbb T^2} v dx = 0\right\}, \quad V= H^1(\mathbb T^2)\cap H
\ee  
representing divergence-free and zero-mean vectors in $L^2(\mathbb T^2)$ and $H^1(\mathbb T^2)$, and let $H^{-1}$ be the dual space of $H^1$. 
Denote by $\Pi f:= f - \nabla\Delta^{-1} \nabla\cdot f$ the Leray-Hodge projection. For simplicity, we often write $L^p$ and $H^k$ instead of $L^p(\mathbb{T}^2)$ and $H^k(\mathbb{T}^2)$ when there is no confusion.

For functions $f\in L^2(\mathbb T^2)$, the Fourier coefficients are denoted by $f_k = \int_{\mathbb T^2} f e^{-ikx} dx$ for $k\in \mathbb Z^2$, and $f$ can be represented as $f=\sum\limits_{k\in \mathbb Z^2} f_k e^{ikx}$. Define the projection $P_N$ as
\[\label{proj}
P_N f = f^N:= \sum\limits_{|k|\leq N} f_k e^{ikx},
\]
which projects a function $f\in L^2(\mathbb T^2)$ into a finite dimensional subspace of $L^2(\mathbb T^2)$.

Fix a stochastic basis $\mathcal{S}=(\Omega, \mathcal{F}, (\mathcal{F}_t)_{t\geq0}, \mathbb{P})$ satisfying the usual conditions. Let $\Uc$ be a real separable Hilbert space with a complete orthonormal basis $\{\mathbf{e}_k\}_{k\geq 1}$, and let $\{W^k\}_{k\geq 1}$ be a family of independent standard Brownian motion so that 
\begin{align}\label{e.L112402}
    \mathbb{W}: = \sum_{k\geq 1} W^k\mathbf{e}_k
\end{align}
is an $\mathcal{F}_t$ adapted and $\Uc$-valued cylindrical Wiener process, for which we mean that the process takes values in a larger Hilbert space $\Uc_0$ such that the embedding from $\Uc$ to $\Uc_0$ is Hilbert-Schmidt, and the trajectories of $\mathbb{W}$ is in $C([0,\infty),\Uc_0)$ $\mathbb P$-a.s., see \cite{da2014stochastic}. Given any Banach space $X$, denote by  $L_2(\Uc, X)$ the space of Hilbert-Schmidt operators $\eta:\Uc\to X$ such that  $\eta_k:=\eta e_k\in X$ for each $k$ and 
$\|\eta\|_{L_2(\Uc, X)}^2:=\sum_{k}\|\eta_k\|_{X}^2<\infty. $
Assume that 
$b, \theta\in L_2(\Uc, L^\infty)$ be divergence-free vector fields and denote
\begin{align*}
    \L_{b_k}f = b_k\cdot \nabla f, \quad \L_{\theta_k}g = \Pi(\theta_k\cdot\nabla g),
\end{align*}
for any regular scalar function $f$ and vector field $g$. Note that since $b_k$ and $\theta_k$ are divergence-free, the dual of $\L_{b_k}, \L_{\theta_k}$ are $-\L_{b_k}, -\L_{\theta_k}$. 
We then interpret equation \eqref{e.w04091} in the corresponding It\^o form 
\begin{align}\label{e.L112403}
    \begin{split}
    dc_1&=\left(D\Delta c_1 - u\cdot \nabla c_1  + D \nabla \cdot (c_1 \nabla \Phi) + \frac{1}{2} \sum\limits_{k} \L_{b_k}^2c_1 \right)dt 
    + \sum\limits_{k}\L_{b_k}c_1d{W}_t^k, 
    \\
    dc_2&=\left(D\Delta c_2 -u\cdot \nabla c_2  - D \nabla \cdot (c_2 \nabla \Phi) +  \frac{1}{2} \sum\limits_{k}\L_{b_k}^2c_2  \right)dt
    + \sum\limits_{k}\L_{b_k}c_2 d{W}_t^k,
    \\
    du&=\left(\nu \Delta u- \Pi\Big(  u\cdot \nabla u + \rho \nabla \Phi\Big)  + \frac{1}{2} \sum\limits_{k}  \L_{\theta_k}^2u \right)dt+  \sum\limits_{k} \L_{\theta_k}u d{W}_t^k, 
    \\
    \nabla\cdot u &=0,
    \\
    -\Delta \Phi &= \rho = c_1 - c_2.
    \end{split}
\end{align}

We now give the definition of a global pathwise solution to the above stochastic NPNS system, which is also called global strong solution in the probabilistic sense. 
\begin{definition}[Pathwise solution]
\label{def:pathwise.sol}
Let $(u_0,c_1(0),c_2(0)) \in L^2\left(\Omega; H\times L^2\times L^2\right)$ be $\Fc_0$-measurable. Then $(c_1,c_2,u)$ is called a global pathwise solution if it is an $\mathcal{F}_t$-adapted process such that for all $T \geq 0$,
\begin{equation}
	\label{eq:solution.regularity}
	(u,c_1,c_2) \in L^2\left(\Omega; C\left([0, T], H\times L^2\times L^2\right) \cap L^2\left(0,T; V\times H^1\times H^1\right)\right),
\end{equation}
and for all $t\in[0, T]$, any function $\phi\in H^1$, and any vector filed $\varphi\in V$, the following identities 
\begin{align}
    \begin{split}
    \langle c_1(t), \phi\rangle &=\langle c_1(0), \phi\rangle+\int_0^t\left(-\langle D\nabla  c_1 - u c_1  +  D c_1 \nabla \Phi,\nabla \phi\rangle - \frac{1}{2} \sum\limits_{k} \langle\L_{b_k}c_1,\L_{b_k}\phi\rangle\right)dr
    \\
    &\hspace{3cm}- \sum\limits_{k}\int_0^t\langle c_1, \L_{b_k}\phi\rangle d{W}_r^k, 
    \\
    \langle c_2(t), \phi\rangle &=\langle c_2(0), \phi\rangle+\int_0^t\left(-\langle D\nabla  c_2 - u c_2  - D c_2 \nabla \Phi,\nabla \phi\rangle - \frac{1}{2} \sum\limits_{k} \langle\L_{b_k}c_2,\L_{b_k}\phi\rangle\right)dr
    \\
    &\hspace{3cm}
    - \sum\limits_{k}\int_0^t\langle c_2, \L_{b_k}\phi\rangle d{W}_r^k, 
    \\
    \langle u(t),\varphi\rangle& =\langle u(0),\varphi\rangle+ \int_0^t\left(-\langle\nu \nabla u, \nabla \varphi\rangle +\langle u\cdot \nabla \varphi, u\rangle - \langle\rho \nabla \Phi,\varphi\rangle  -\frac{1}{2} \sum\limits_{k} \langle \L_{\theta_k}u , \L_{\theta_k}\varphi\rangle\right)dr
    \\
    &\hspace{3cm}-\sum\limits_{k}\int_0^t \langle u,\L_{\theta_k}\varphi\rangle d{W}_r^k, 
    \\
    \nabla\cdot u &=0,
    \\
    -\Delta \Phi &= \rho = c_1 - c_2,
    \end{split}
\end{align}

hold $\mathbb{P}$-a.s.. 
\end{definition}

\subsection{Main results}
Our first result is on the global well-posedness of the stochastic NPNS system \eqref{e.w04091} driven by a general transport noise.   

\begin{theorem}[Global well-posedness]
\label{t.w04092}
Let $(u_0,c_1(0),c_2(0)) \in L^2\left(\Omega; H\times L^2\times L^2\right)$ be $\Fc_0$-measurable random variables such that $c_1(0),c_2(0)\geq 0$ $\mathbb P$-a.s.. Assume $b, \theta\in L_2(\Uc, L^\infty)$. Then there exists a unique global pathwise solution to the system \eqref{e.w04091} in the sense of Definition \ref{def:pathwise.sol}.
\end{theorem}

The second main result is on the enhanced dissipation by transport noise. Before stating the result we recall the special type of noise as in \cite{luo2023enhanced}.
Unlike \eqref{e.L112402} where the index $k\in\mathbb N$, here we
%In the settings of the noise as in \eqref{e.L112402}, 
let $\left\{\mathbf{e}_k, k\in\mathbb{Z}_0^2\right\}$ be a basis of the Hilbert space $\Uc$, and denote by 
$
    \mathbb{W}:=\sum_{k\in\mathbb{Z}_0^2} W^k \mathbf{e}_k,
$
where $\{W^k\}_{k\in\mathbb{Z}_0^2}$ is a sequence of independent standard complex Brownian motion satisfying 
\begin{align*}
    \overline{W^k} = W^{-k}, \quad [W^k, W^{l}]_t = 2t\delta_{k, -l}, \quad k,l \in \mathbb{Z}_0^2. 
\end{align*}
Assume that the coefficients $b, \theta$ of the noise have the following same structure:
\begin{align}\label{e.L122101}
    b = \theta = (\sqrt{2\kappa}\zeta_k\sigma_k)_{k\in\mathbb{Z}_0^2}\in \ell^2=\ell^2(\mathbb{Z}_0^2). 
\end{align}
Here $\kappa>0$ is the noise intensity, and 
\begin{align}\label{e.L112504}
    \zeta_k=\zeta_k^N = \frac{1}{\Lambda^N}\frac{1}{|k|^{\gamma}}\mathbf{1}_{N\leq |k| \leq 2N},
\end{align}
where $\gamma>0$ is some positive constant,  $N\in \mathbb{N}$, and $\Lambda_N$ is a normalizing constant to ensure $\|\zeta\|_{\ell^2}=1$ where $\zeta = (\zeta_k)\in\ell^2$. For more details on the choice of $\gamma$, $N$, and $\Lambda_N$, we refer the readers to \cite{luo2023enhanced}.
Furthermore $\{\sigma_k\}_{k\in\mathbb{Z}_0^2}$ is a sequence of divergence-free vector fields on $\mathbb{T}^2$ such that $\sigma_k = a_k e_k$ with $e_k$ being the usual complex basis of $L^2(\mathbb{T}^2,\mathbb{C})$ and 
\begin{align}\label{e.L112512}
    a_k = \begin{cases}
        \frac{k^{\perp}}{|k|}, & k\in\mathbb{Z}_+^2,\\
        -\frac{k^{\perp}}{|k|}, & k\in\mathbb{Z}_-^2,
    \end{cases}
\end{align}
where $\mathbb{Z}_0^2=\mathbb{Z}_+^2\cup \mathbb{Z}_-^2$ is a partition of $\mathbb{Z}_0^2$ such that $\mathbb{Z}_+^2=-\mathbb{Z}_-^2$. Let $\gamma_0>0$ be the constant from the Sobolev embedding
\[\|f\|_{L^{\infty}}\leq \gamma_0\|f\|_{H^2}\]
and for $i=1,2$, we denote the spatial average of the concentration $c_i$ by $\bar{c}_i$, which is preserved by the NPNS system. 
\begin{theorem}[Enhanced dissipation]
\label{t.w04093}
Given any $L, \lambda>0$ and $p\geq 1$, there exists a pair $(\kappa, \zeta)$ of the noise coefficient \eqref{e.L122101} such that for any $\Fc_0$-measurable initial data $(u_0,c_1(0),c_2(0)) \in L^2\left(\Omega; H\times L^2\times L^2\right)$ satisfying $c_1(0),c_2(0)\geq 0$, 
    $$\|c_1(0)-\bar{c}_1\|_{L^2}^2+\|c_2(0)-\bar{c}_2\|_{L^2}^2<\frac12\nu D\gamma_0^{-2},$$
    and $$\|u(0)\|_{L^2}+\|c_1(0)\|_{L^2}+\|c_2(0)\|_{L^2}\leq L,$$
the solution has $\lambda$ as the decay rate, i.e., 
\begin{align*}
    \|u(t)\|_{L^2}^2+\|c_1(t)-\bar{c}_1\|^2+\|c_2(t)-\bar{c}_2\|^2\leq C e^{-\lambda t}(\|u(0)\|_{L^2}^2+\|c_1(0)-\bar{c}_1\|^2+\|c_2(0)-\bar{c}_2\|^2),
\end{align*}
almost surely. Here $C>0$ is a random constant with finite $p$-th moment. 
\end{theorem}

\section{Analysis of the cut-off system}\label{cut-off}

For notational convenience, throughout this section we denote $U=(u,c_1,c_2)$, and 
\begin{align}\label{e.L122301}
    \mch = H\times L^2\times L^2, \quad \mcv=V\times H^1\times H^1,
\end{align}
with the corresponding norm $\|\cdot\|_{L^2}$ for $\mch$ as 
% (or $U\in\mcv$)
\[\|U\|_{L^2}^2 = \|u\|_{L^2}^2+\|c_1\|_{L^2}^2+\|c_2\|_{L^2}^2, \, \text{ for } U\in\mch.\]
%\text{ or } \|U\|_{H^1} = \|u\|_{H^1}^2+\|c_1\|_{H^1}^2+\|c_2\|_{H^1}^2.
We also denote 
\begin{align}\label{e.L122303}
    \L_k = (\L_{\theta_k}, \L_{b_k}, \L_{b_k}), \, k\geq 1. 
\end{align}
Let the function $\tau_\eta\in C^\infty(\mathbb R)$ be a non-increasing cut-off function such that 
\begin{equation}\label{eqn:rho}
    \mathbbm{1}_{[0, \frac{\eta}2]} \leq \tau_\eta \leq \mathbbm{1}_{[0, \eta]}.
\end{equation}
For simplicity we use $\tau_\eta$ to represent $\tau_\eta(\|c_1\|_{L^2} +\|c_2\|_{L^2})$ below. 

In order to show the existence of solutions, we first consider the following cut-off system of \eqref{e.L112403}, 
\begin{align}\label{sys:modified}
    \begin{split}
    dc_1&=\left(D\Delta c_1 - u\cdot \nabla c_1  + D \tau_\eta\nabla \cdot (c_1 \nabla \Phi) + \frac{1}{2} \sum\limits_{k} \L_{b_k}^2c_1 \right)dt 
    + \sum\limits_{k}\L_{b_k}c_1d{W}_t^k, 
    \\
    dc_2&=\left(D\Delta c_2 -u\cdot \nabla c_2  - D \tau_\eta\nabla \cdot (c_2 \nabla \Phi) +  \frac{1}{2} \sum\limits_{k}\L_{b_k}^2c_2  \right)dt
    + \sum\limits_{k}\L_{b_k}c_2 d{W}_t^k,
    \\
    du&=\left(\nu \Delta u- \Pi\Big(  u\cdot \nabla u + \tau_{\eta}\rho \nabla \Phi\Big)  + \frac{1}{2} \sum\limits_{k}  \L_{\theta_k}^2u \right)dt+  \sum\limits_{k} \L_{\theta_k}u d{W}_t^k, 
    \\
    \nabla\cdot u &=0,
    \\
    -\Delta \Phi &= \rho = c_1 - c_2,
    \end{split}
\end{align}
The aim of this section is to show the following global well-posedness of this cut-off system.

\begin{theorem}\label{t.L122201}
For any $\mathcal{F}_0$-measurable initial data $U_0=\left(u_0, c_1(0), c_2(0)\right) \in L^2\left(\Omega,\mch\right)$, there is a unique global pathwise solution $U^{\eta}$ of the cut-off system \eqref{sys:modified}. 
\end{theorem}
\begin{proof}
In view of the Proposition \ref{p.L122201} in Section \ref{s.L122201} on the existence of a martingale solution, and Proposition \ref{p.w05241} in Section \ref{s.L122202} on the pathwise uniqueness, we immediately obtain the existence of a unique global pathwise solution of \eqref{sys:modified} for initial data $U_0\in L^4(\Omega,\mch)$, by invoking the well-known Yamada-Watanabe theorem (see \cite{kurtz2007yamada} Theorem 3.14). To deal with initial data $U_0\in L^2(\Omega,\mch)$, for each $k\geq 0$ we set $A_{k}=\{\omega\in\Omega: k\leq \|U_0\|_{L^2}\leq k+1\}$. Then each $U_0\mathbf{1}_{A_k}\in L^4(\Omega,\mch)$ and hence we have a unique global pathwise solution $U_k^{\eta}$ to \eqref{sys:modified}. It is straightforward to check that 
\[U^{\eta} = \sum_{k}U_k^{\eta}\mathbf{1}_{A_k}\]
is a global pathwise solution of \eqref{sys:modified} with initial data $U_0\in L^2(\Omega,\mch)$. 
\end{proof}
Thus it remains to prove Proposition \ref{p.L122201} and Proposition \ref{p.w05241}. We first give the definition of a martingale solution to system \eqref{sys:modified}. 

\begin{definition}[Martingale solution]\label{def.L121901}
 Given any $\Fc_0$ measurable initial data $U_0=(u_0,c_1(0),c_2(0))\in L^2(\Omega,\mch)$, a triple $(\widetilde{\Sc},\widetilde{U}, \widetilde{\mathbb{W}})$ is called a global martingale solution of the cut-off system \eqref{sys:modified} with initial data $U_0$ if $\widetilde{\Sc} = \left(\widetilde{\Omega}, \widetilde{\Fc}, \widetilde{\mathcal{F}}_t, \widetilde{\Pb}\right)$ is a stochastic basis, $\widetilde{\mathbb{W}}=\sum_{k}\widetilde{W}^k\mathbf{e}_k$ is an $\widetilde{\mathcal{F}}_t$-adapted cylindrical Wiener process in $\Uc$, and for any $T>0$,
\begin{align*}
 \widetilde{U}=(\widetilde{u},\widetilde{c}_1, \widetilde{c}_2)\in L^2\left(\widetilde{\Omega}; C\left([0, T]; \mch\right) \cap L^2\left(0,T; \mcv\right)\right)    
\end{align*}
is  $\widetilde{\mathcal{F}}_t$-adapted such that the law of $\widetilde{U}(0)$  equals the law of $U_0$, and for all $t\in[0, T]$, any function $\phi\in H^1$, and any vector field $\varphi\in V$, the following identities 
\begin{align}\label{eq:solution.def}
    \begin{split}
    \langle \widetilde{c}_1(t), \phi\rangle &=\langle \widetilde{c}_1(0), \phi\rangle+\int_0^t\left(-\langle D\nabla  \widetilde{c}_1 - \widetilde{u} \widetilde{c}_1  +  D\widetilde{\tau}_{\eta} \widetilde{c}_1 \nabla \widetilde{\Phi},\nabla \phi\rangle - \frac{1}{2} \sum\limits_{k} \langle\L_{b_k}\widetilde{c}_1,\L_{b_k}\phi\rangle\right)ds 
    \\
    &\hspace{3cm}- \sum\limits_{k}\int_0^t\langle \widetilde{c}_1, \L_{b_k}\phi\rangle d\widetilde{W}_r^k, 
    \\
    \langle \widetilde{c}_2(t), \phi\rangle &=\langle \widetilde{c}_2(0), \phi\rangle+\int_0^t\left(-\langle D\nabla  \widetilde{c}_2 - \widetilde{u} \widetilde{c}_2  - D \widetilde{\tau}_{\eta}\widetilde{c}_2 \nabla \widetilde{\Phi},\nabla \phi\rangle - \frac{1}{2} \sum\limits_{k} \langle\L_{b_k}\widetilde{c}_2,\L_{b_k}\phi\rangle\right)ds 
    \\
    &\hspace{3cm}
    - \sum\limits_{k}\int_0^t\langle \widetilde{c}_2, \L_{b_k}\phi\rangle d\widetilde{W}_r^k, 
    \\
    \langle \widetilde{u}(t),\varphi\rangle& =\langle \widetilde{u}(0),\varphi\rangle+ \int_0^t\left(-\langle\nu \nabla \widetilde u, \nabla \varphi\rangle +\langle \widetilde{u}\cdot \nabla\varphi, \widetilde{u}\rangle -\langle \widetilde{\tau}_{\eta}\widetilde{\rho} \nabla \widetilde{\Phi},\varphi\rangle  -\frac{1}{2} \sum\limits_{k} \langle \L_{\theta_k}\widetilde{u} , \L_{\theta_k}\varphi\rangle\right)dt
    \\
    &\hspace{3cm}-  \sum\limits_{k} \int_0^t\langle \widetilde{u},\L_{\theta_k}\varphi\rangle d\widetilde{W}_r^k, 
    \\
    \nabla\cdot \widetilde{u} &=0,
    \\
    -\Delta \widetilde{\Phi} &= \widetilde{\rho} = \widetilde{c}_1 - \widetilde{c}_2,
    \end{split}
\end{align}
hold $\widetilde{\mathbb{P}}$-a.s., where 
$\widetilde{\tau}_{\eta}=(\|\widetilde{c_1}\|_{L^2}+ \|\widetilde{c_2}\|_{L^2}).$
\end{definition}
A major distinction between a martingale solution and a pathwise solution, as outlined in Definition \eqref{def:pathwise.sol}, lies in the incorporation of the stochastic basis and the driving noise into the solution itself in the case of martingale solutions.

\subsection{Uniform estimates of the Galerkin system and compactness}
This subsection is devoted to establishing the uniform estimates needed to pass to the limit in the following Galerkin approximation scheme for the cut-off system \eqref{sys:modified}: 
\begin{align}\label{sys:galerkin}
    \begin{split}
    dc_1^n&=D\Delta c_1^n +P_n\left(- u^n\cdot \nabla c_1^n  + D \tau_\eta^{n}\nabla \cdot (c_1^n \nabla \Phi^n) + \frac{1}{2} \sum\limits_{k} \L_{b_k}P_n \L_{b_k} c_1^n \right)dt 
    + \sum\limits_{k}P_n(\L_{b_k}c_1^n)d{W}_t^k, 
    \\
    dc_2^n&=D\Delta c_2^n+ P_n\left(-u^n\cdot \nabla c_2^n  - D \tau_\eta^{n}\nabla \cdot (c_2^n \nabla \Phi^n) +  \frac{1}{2} \sum\limits_{k}\L_{b_k}P_n \L_{b_k} c_2^n  \right)dt
    + \sum\limits_{k}P_n(\L_{b_k}c_2^n) d{W}_t^k,
    \\
    du^n&=\nu \Delta u^n + P_n\left(-\Pi\Big(  u^n\cdot \nabla u^n + \tau_\eta^{n}\rho^n \nabla \Phi^n\Big)  + \frac{1}{2} \sum\limits_{k}  \L_{\theta_k} P_n \L_{\theta_k} u^n \right)dt+  \sum\limits_{k}P_n(\L_{\theta_k}u^n)d{W}_t^k, 
    \\
    \nabla\cdot u^n &=0,
    \\
    -\Delta \Phi^n &= \rho^n = c_1^n - c_2^n,
    \end{split}
\end{align}
with initial data $U_0^n:=P_n(u_0,c_1(0),c_2(0))$ and 
\[\tau_\eta^{n}=\tau_\eta(\|c_1^n\|_{L^2} +\|c_2^n\|_{L^2}).\]
For the above stochastic differential equations, it is standard (see for example \cite{prevot2007concise}) that there exists a unique global pathwise solution  
\[U^n=(u^n,c_1^n,c_2^n)\in L^2\left(\Omega; C\left([0, T]; \mch\right) \cap L^2\left(0,T; \mcv\right)\right),  \text{ for any } T>0.\]
We next establish the conservation properties of the solutions. 
\begin{lemma}\label{lem:estimate}
For $p\geq 2$, suppose that $U(0)=\left(u_0, c_1(0), c_2(0)\right) \in L^p\left(\Omega ; \mch\right)$ is $\Fc_0$-measurable. Let $U^n=(u^n, c_1^n, c_2^n)$ be the solution to system \eqref{sys:galerkin}. Then the following conclusions hold.  
\begin{enumerate}[(I)]
	\item For any time $T>0$, 
 \begin{align}\label{eq.L121801}
     \sup\limits_{n\in \N}	 \Big[\sup\limits_{s\in[0,T]} \|U^n(s)\|_{L^2}^2 +  \int_0^T \|\nabla U^n(s)\|_{L^2}^2  dr \Big]  
     \leq C_{T,\eta}(1 +\|U(0)\|_{L^2}^2)
 \end{align}
almost surely, where $C_{T,\eta}>0$ is a deterministic constant depending on $T,\eta$;
    \item For  $\alpha\in [0,1/2)$, we have
    \begin{align}\label{eq.L122304}
    \mathbb E\left\|\int_0^{t} \sum\limits_{k} P_n(\L_k U^n) dW^k_r \right\|^p_{W^{\alpha,p}(0,T;\mcv^{-1})} \leq C_{p,T,\eta}(1+\mathbb E\|U(0)\|_{L^2}^p), 
\end{align} 
where $\mcv^{-1}$ is the dual space of $\mcv$ in \eqref{e.L122301} and the operator $\L_k$ is defined in \eqref{e.L122303}.  
    \item Furthermore, 
    \begin{align}\label{eq.L122305}
        \mathbb E\left\|U^n(t) - \int_0^{t} \sum\limits_{k} P_n(\L_k U^n) dW^k_r \right\|_{W^{1,2}(0,T;\mcv^{-1})}\leq C_{T,\eta}(1+\mathbb E\|U(0)\|_{L^2}^2). 
    \end{align}
\end{enumerate}
\end{lemma}
\begin{proof}
{\bf (I)}
Applying It\^o's formula to the equations of $c^n_1, c^n_2, u^n$ and using the divergence-free property of $u$, $b$ and $\theta$, we obtain
\begin{align}
    &\|c_1^n(t)\|_{L^2}^2 + 2D \int_0^t \|\nabla c_1^n(r)\|_{L^2}^2 dr = \|c^n_1(0)\|_{L^2}^2 + 2D \int_0^t\tau_\eta^n \langle \nabla\cdot (c_1^n \nabla \Phi^n), c_1^n\rangle dr, \label{eqn:wp-1}
    \\
    &\|c_2^n(t)\|_{L^2}^2 + 2D \int_0^t \|\nabla c_2^n(r)\|_{L^2}^2 dr = \|c^n_2(0)\|_{L^2}^2 -2D \int_0^t \tau_\eta^n\langle \nabla\cdot (c_2^n \nabla \Phi^n), c_2^n\rangle dr,\label{eqn:wp-2}
    \\
    &\|u^n(t)\|_{L^2}^2 + 2\nu \int_0^t \|\nabla u^n(r)\|_{L^2}^2 dr = \|u^n(0)\|_{L^2}^2 -2 \int_0^t\tau_\eta^n \langle \rho^n \nabla\Phi^n, u^n\rangle dr,\label{eqn:wp-3}
\end{align}
where we used the cancellation  
$$\langle \L_{b_k} P_n \L_{b_k} c_i^n, c_i^n\rangle  + \|P_n \L_{b_k}c_i^n\|_{L^2}^2 = 0, \quad \langle \L_{\theta_k} P_n \L_{\theta_k} u^n, u^n\rangle + \|P_n\L_{\theta_k} u^n\|_{L^2}^2=0,$$
since the dual of $\L_{b_k}$ and $\L_{\theta_k}$ are $-\L_{b_k}$ and $-\L_{\theta_k}$. 
For the nonlinear terms in \eqref{eqn:wp-1}, by H\"older, Ladyzhenskaya and Young's inequalities together with the standard elliptic estimates, we have 
\begin{align*}
    &\left|2D \int_0^t \tau_\eta^n\langle \nabla\cdot (c_1^n \nabla \Phi^n), c_1^n\rangle  dr \right|
    \\
    \leq & C \int_0^t\tau_\eta^n \|\nabla c_1^n\|_{L^2} \|c_1^n\|_{L^4} \|\nabla \Phi^n\|_{L^4} dr
    \\
    \leq & C \int_0^t \tau_\eta^n( \|\nabla c_1^n\|_{L^2}^{\frac32} \|c_1^n\|_{L^2}^{\frac12} + \|\nabla c_1^n\|_{L^2} \|c_1^n\|_{L^2})(\|c_1^n\|_{L^2} +\|c_2^n\|_{L^2}) dr
    \\
    \leq & D\int_0^T \|\nabla c_1^n\|_{L^2}^{2} dr + C_{T,\eta},
\end{align*}
where the last step we use the property of the cut-off function $\tau_\eta$, and $C_{T,\eta}$ is some deterministic constant depending on $T$ and $\eta$. One can get a similar estimate for the nonlinear term in \eqref{eqn:wp-2} and conclude that
\begin{align}\label{eqn:wp-est-1}
    \sum\limits_{i=1}^2 \|c_i^n(t)\|_{L^2}^2 + D \sum\limits_{i=1}^2 \int_0^t \|\nabla c_i^n(r)\|_{L^2}^2 dr \leq \sum\limits_{i=1}^2 \|c_i^n(0)\|_{L^2}^2 + C_{T,\eta} \leq \sum\limits_{i=1}^2 \|c_i(0)\|_{L^2}^2 + C_{T,\eta}
\end{align}
almost surely. Notice that $u$ has zero mean, one has
\begin{align*}
    \left|\int_0^t \tau_\eta^n\langle \rho^n \nabla\Phi^n , u^n\rangle dr \right| \leq  &2\int_0^t \tau_\eta^n\|\rho^n\|_{L^2} \|\nabla\Phi^n\|_{L^4} \|u^n\|_{L^4} dr 
    \\
    \leq &C \int_0^t\tau_\eta^n\|\rho^n\|^2_{L^2}  \|\nabla u^n\|_{L^2} dr
    \\
    \leq &C\int_0^t\tau_\eta^{n}(\|c_1^n\|^4_{L^2} + \|c_2^n\|^4_{L^2})dr + \nu \int_0^t\|\nabla u^n\|_{L^2}^2 dr
    \\
    \leq &  \nu \int_0^T\|\nabla u^n\|_{L^2}^2 dr + C_{T,\eta}.
\end{align*}
This together with \eqref{eqn:wp-3} yield that
\begin{align}\label{eqn:wp-est-2}
    \|u^n(t)\|_{L^2}^2 + \nu \int_0^t \|\nabla u^n(r)\|^2_{L^2} dr \leq \|u^n(0)\|^2_{L^2}  + C_{T,\eta} 
    \leq\|u_0\|^2_{L^2}  + C_{T,\eta}
\end{align}
almost surely. Combining \eqref{eqn:wp-est-1} and \eqref{eqn:wp-est-2} we conclude result (I).

{\bf (II)} By raising to the power $p$, \eqref{eqn:wp-est-1} and \eqref{eqn:wp-est-2} also yield that
\begin{align}\label{eq.L122501}
    \|U^n\|_{L^2}^p\leq C_{p,T,\eta}(1+\|U(0)\|_{L^2}^p)
\end{align}
almost surely.
By taking the time integral from $0$ to $T$ and then taking the expectation, we obtain 
\begin{align}\label{eqn:wp-est-3}
    \mathbb E \left[\int_0^T 
    \|U^n(t)\|_{L^2}^pdt \right] \leq C_{p,T,\eta}(1+\mathbb E\|U(0)\|_{L^2}^p). 
\end{align}
By the Burkholder-Davis-Gundy inequality \cite{flandoli1995martingale}, we have
\begin{align*}
    \mathbb E\left\|\int_0^t \sum\limits_{k} P_n(\L_k U^n) dW^k_r \right\|^p_{W^{\alpha,p}(0,T;H^{-1})}
    \leq &C \mathbb E \int_0^T \left(\sum\limits_{k}\|P_n(\L_kU^n)\|_{H^{-1}}^2\right)^{p/2} dt 
    \\
    \leq &C \mathbb E \int_0^T \left(\sum\limits_{k}\left(\|\theta_k\|^2_{L^{\infty}}+2\|b_k\|^2_{L^{\infty}}\right)\right)^{\frac p2} \|U^n\|_{L^2}^p dt. 
\end{align*} 
The conclusion (II) then follows from \eqref{eqn:wp-est-3} and the assumption that $\theta, b \in L_2(\Uc, L^\infty)$. 

{\bf (III)} Note that 
\begin{align*}
    &u^n(t) - \int_0^t \sum_{k} P_n(\L_{\theta_k} u^n) dW^k_r
    \\
    =&u^n(0) + \int_0^t \left(\nu \Delta u^n + P_n\left( -\Pi( u^n\cdot \nabla u^n + \tau_\eta^{n}\rho^n \nabla \Phi^n ) +  \frac12\sum_{k}\L_{\theta_k}^2 u^n \right) \right)dr.
\end{align*}
By Ladyzhenskaya inequality and standard elliptic estimates, we have
\begin{align}
    &\left\|u^n(t) - \int_0^t \sum_{k} P_n(\L_{\theta_k} u^n) dW^k_r\right\|_{W^{1,2}(0,T;H^{-1})} \nonumber
    \\
    \leq &\|u_0\|_{L^2} + C_T\int_0^T \left(\|\Delta u^n\|_{H^{-1}} + \|u^n\cdot \nabla u^n\|_{H^{-1}} + \|\tau_\eta^{n}\rho^n \nabla \Phi^n\|_{H^{-1}} + \frac12\sum_{k}\|\L_{\theta_k}^2 u^n\|_{H^{-1}} \right) dt \nonumber
    \\
    \leq &\|u_0\|_{L^2} + C_T\int_0^T \left(\|\nabla u^n\|_{L^2} + \|u^n\|_{L^2} \|\nabla u^n\|_{L^2} + C_{T,\eta}+ \sum_{k}\|\theta_k\|_{L^{\infty}}^2 \|\nabla u^n\|_{L^2}\right) dt\nonumber
    \\ \leq &C_{T,\eta}(1+\|U(0)\|_{L^2}^2),\label{eqn:wp-est-4}
\end{align}
where in the last step we used Cauchy–Schwarz inequality and estimate \eqref{eq.L121801}. 
For $i=1,2$, we can bound $c_i^n - \int_0^{\cdot} \sum_k P_n (\L_{b_k} c_i^n) dW^k_r$ similarly, and this concludes the proof of (III). 
\end{proof}

The above uniform-in-$n$ estimates enable us to show the tightness of  the sequence of laws of the solutions to the Galerkin scheme, whose limit provides a martingale solution to the cut-off system \eqref{sys:modified}. We next define the path space 
\begin{align}
    \mathcal{X} = \mch\times \left(L^2\left(0,T;\mch\right)\cap C\left([0, T]; H^{-2}\right)\right) \times C([0, T]; \Uc_0). 
\end{align}
Given any initial data $U_0\in L^2(\Omega, \mch)$, let $\mu_0^n$ be the law of $U_0^n=P_nU_0$ on $\mch$, and $\mu_U^n$ be the law of the corresponding solution $U^n$ of \eqref{sys:galerkin} on $L^2\left(0,T; \mch\right)\cap C\left([0, T]; H^{-2}\right)$, also $\mu_{\mathbb W}$ be the law of the cylindrical Wiener processes  $\mathbb{W}$ from \eqref{e.L112402} on $C([0, T], \Uc_0)$, and $\mu^n$ be their joint law in the path space $\mathcal{X}$.

\begin{proposition}\label{p.L122301}
    Let $p>2$. Given any initial data $U_0\in L^p(\Omega, \mch)$, the sequence of joint laws $\{\mu^n\}_{n\geq 1}$ is tight on the path space $\mathcal{X}$. 
\end{proposition}
\begin{proof}
Note that $\mu_{\mathbb W}$ is a single distribution, hence it is tight. Since $P_nU_0\to U_0$ in $L^2(\Omega, \mch)$, the convergence also holds in distribution. Hence the sequence $\mu_0^n$ is tight by Prokhorov's theorem. It remains to show the tightness of the sequence $\mu_U^n$. 

Given any $R>0$, let $B_R^1$ be the closed ball defined as
\[
B_R^1 := \left\{U\in L^2(0,T;\mcv)\cap W^{\frac14,2}(0,T;\mcv^{-1}): \|U^n\|_{L^2(0,T;\mcv)}^2+\|U^n\|_{W^{\frac14,2}(0,T;\mcv^{-1})}^2\leq R^2\right\}.
\]

Denote by $B_R^{2,1}$ the closed ball of radius $R$ in $W^{1,2}(0,T;\mcv^{-1})$ and $B_R^{2,2}$ the closed ball of radius $R$ in $W^{\alpha,p}(0,T;\mcv^{-1})$ for some $\alpha\in(0, \frac12)$ such that $\alpha p>1$, and denote $B_R^2 =B_R^{2,1}+ B_R^{2,2}$. Then by the Markov inequality and estimates \eqref{eq.L121801}, \eqref{eq.L122304} and \eqref{eq.L122305}  we have 
\begin{align*}
    \begin{split}
        \mu_U^n\left((B_R^1)^c\right) &= \mathbb{P}\left(\|U^n\|_{L^2(0,T;\mcv)}^2+\|U^n\|_{W^{\frac14,2}(0,T;\mcv^{-1})}^2\geq R^2 \right)\\
        &\leq \mathbb{P}\left(\|U^n\|_{L^2(0,T;\mcv)}\geq \frac{R}{\sqrt2}\right)+\mathbb{P}\left(\|U^n\|_{W^{\frac14,2}(0,T;\mcv^{-1})}\geq \frac{R}{\sqrt2}\right)\\
        &\leq \frac{C}{R}(1+\E\|U(0)\|_{L^2}^2),
    \end{split}
\end{align*}
and 
\begin{align*}
    \begin{split}
     \mu_U^n\left((B_R^2)^c\right) \leq& \mathbb{P}\left(\left\|U^n(t) - \int_0^{t} \sum\limits_{k} P_n(\L_k U^n) dW^k_r \right\|_{W^{1,2}(0,T;\mcv^{-1})}\geq R\right) \\
     &\qquad+ \mathbb{P}\left(\left\|\int_0^{t} \sum\limits_{k} P_n(\L_k U^n) dW^k_r \right\|^p_{W^{\alpha,p}(0,T;\mcv^{-1})}\geq R^p\right)\\
     &\leq \frac{C}{R}(1+\E\|U(0)\|_{L^2}^p).
    \end{split}
\end{align*}
By Aubin-Lions compactness theorem \cite{flandoli1995martingale}, we have that $L^2(0,T;\mcv)\cap W^{\frac14,2}(0,T;\mcv^{-1})$ is compactly embedded in $L^2(0,T;\mch)$, while both $W^{1,2}(0,T;\mcv^{-1})$ and $W^{\alpha,p}(0,T;\mcv^{-1})$ are compactly embedded in $C([0,T];H^{-2})$ since $\alpha p>1$. Hence $B_R=B_R^1\cap B_R^2$ is compact in $L^2(0,T;\mch)\cap C([0,T];H^{-2})$ for every $R>0$, and 
\begin{align*}
    \mu_U^n(B_R^c)\leq \mu_U^n\left((B_R^1)^c\right)+\mu_U^n\left((B_R^2)^c\right) \leq \frac{C}{R}(1+\E\|U(0)\|_{L^2}^p). 
\end{align*}
Therefore the sequence $\mu_U^n$ is tight on $L^2(0,T;\mch)\cap C([0,T];H^{-2})$. Thus the sequence of joint laws $\mu^n$ is tight on $\mathcal{X}$. 
\end{proof}

By Proposition \ref{p.L122301} and Prokhorov's theorem we readily obtain the existence of a subsequence $\mu^{n_j}$ of the sequence of joint laws $\mu^n$ that converges weakly to some probability measure $\mu$ on the path space $\mathcal{X}$.  Thus by the Skorokhod representation theorem \cite{billingsley2013convergence}, we have the following 

\begin{proposition}\label{p.L121801}
Let $p>2$ and the initial data $U_0\in L^p(\Omega, \mch)$.  Then there exists a probability space $(\widetilde{\Omega}, \widetilde{\Fc}, \widetilde{\Pb})$, a subsequence $n_j \to \infty$ and a sequence of $\Xc$-valued random variables $(\widetilde{U}_0^{n_j}, \widetilde{U}^{n_j},\widetilde{\mathbb W}^{n_j})$ with $(\widetilde{U}_0, \widetilde{U},\widetilde{\mathbb W})$ defined on $(\widetilde{\Omega}, \widetilde{\Fc}, \widetilde{\Pb})$ such that $\widetilde{\mathbb W}^{n_j}=\sum_k\widetilde{ W}^{k,n_j}\mathbf{e}_k$ and $\widetilde{\mathbb W}=\sum_k\widetilde{W}^k\mathbf{e}_k$ are cylindrical  Wiener processes and the following conclusions hold. 
\begin{enumerate}
    \item The law of $(\widetilde{U}_0, \widetilde{U},\widetilde{\mathbb W})$ is the limit measure $\mu$, and the law of $(\widetilde{U}_0^{n_j}, \widetilde{U}^{n_j},\widetilde{\mathbb W}^{n_j})$ coincides with $\mu^{n_j}$ for each $n_j$;
    \item $(\widetilde{U}_0^{n_j}, \widetilde{U}^{n_j},\widetilde{\mathbb W}^{n_j})$ converges $\widetilde{\Pb}$ almost surely  to $(\widetilde{U}_0, \widetilde{U},\widetilde{\mathbb W})$ in $\Xc$; 
    \item Each triple $(\widetilde{U}_0^{n_j}, \widetilde{U}^{n_j},\widetilde{\mathbb W}^{n_j})$ with $\widetilde{U}^{n_j}=(\widetilde{u}^{n_j}, \widetilde{c_1}^{n_j},\widetilde{c_2}^{n_j})$ being a martingale solution to the cut-off system \eqref{sys:galerkin} with $n=n_j$ in the sense of Definition \ref{def.L121901}.
    In particular, the solution $\tdu^{n_j}$ satisfies similar uniform estimates as $U^{n_j}$ in Lemma \ref{lem:estimate}.  
\end{enumerate}
\end{proposition}

\subsection{Global martingale solutions of the cut-off system}\label{s.L122201}

In Proposition \ref{p.L121801} we obtain a sequence of martingale solutions to the Galerkin scheme of the cut-off system with an almost-sure limit in the path space $\mathcal{X}$. This subsection is devoted to showing that this limit is a martingale solution to the cut-off system.

\begin{proposition}\label{p.L122201}
Let $p\geq 4$ and the initial data $U_0\in L^p(\Omega, \mch)$. The limit $(\widetilde{U}_0, \widetilde{U},\widetilde{\mathbb{W}})$  obtained in Proposition \ref{p.L121801} together with the stochastic basis $(\widetilde{\Omega}, \widetilde{\Fc}, \{\widetilde{\Fc}_t\}_{t \geq 0}, \widetilde{\Pb})$ is a global martingale solution of the cut-off system \eqref{sys:modified},  where $\widetilde{\Fc}_t$ is the completion of $\sigma(\widetilde{\mathbb{W}}(s),\tdu(s):s\leq t)$. 
\end{proposition}
\begin{proof}
We first claim that 
\begin{align}\label{eq.L121902}
    \tdu \in L^2(\widetilde{\Omega};L^2(0,T; \mcv)\cap L^{\infty}(0,T; \mch)).
\end{align}
Estimate \eqref{eq.L121801} and the Banach–Alaoglu theorem imply the existence of $U_1\in L^2(\widetilde{\Omega};L^2(0,T; \mcv))$ and $U_2\in L^2(\widetilde{\Omega}; L^{\infty}(0,T; \mch))$ such that 
\begin{align}
    &\tdu^{n_j}\weakconv U_1 \text{ in } L^2(\widetilde{\Omega};L^2(0,T; \mcv)),\label{eq.L122503}
    \\
    &\tdu^{n_j}\weakconv^* U_2 \text{ in } L^2(\widetilde{\Omega};L^{\infty}(0,T; \mch)). \label{eq.L122502}
\end{align}

Next, Proposition \ref{p.L121801} implies that $\tdu^{n_j}\to \tdu$ almost surely in $C([0,T];H^{-2})$. As we have assumed $p\geq 4>2$, Lemma \ref{lem:estimate} together with the Vitali convergence theorem yield the strong convergence 
\begin{align}\label{eq.L122504}
    \tdu^{n_j}\to \tdu \, \text{ in } L^2(\widetilde{\Omega};L^{\infty}(0,T; H^{-2})).
\end{align}
In view of \eqref{eq.L122502}, this shows that $\tdu=U_2\in L^2(\widetilde{\Omega}; L^{\infty}(0,T; \mch))$. Also \eqref{eq.L122504} and \eqref{eq.L122503} imply that 
\[\widetilde{\E}\int_0^T\mathbf{1}_A\langle U_1,\varphi\rangle dr=\widetilde{\E}\int_0^T\mathbf{1}_A\langle \tdu,\varphi\rangle dr,\]
for any measurable set $A\subseteq [0,T]\times\widetilde{\Omega}$ and $\varphi\in H^2$. Thus $\tdu=U_1\in L^2(\widetilde{\Omega};L^2(0,T; \mcv))$ and the claim \eqref{eq.L121902} is proved. The proof then is divided into the following three steps.

{\it Step 1: Convergence of the linear and stochastic terms.} By Proposition \ref{p.L121801}, we know 
\[\tdu^{n_j}\to \tdu\, \text{ in } L^2(0,T;\mch) \,\, \text{and}\,\, \tdu^{n_j}(0)\to \tdu(0)\, \text{ in } \mch \quad a.s..\]
% Since we have assumed $p\geq 4$, by 
Then immediately we have
\begin{align*}
    \langle \tdu^{n_j}(0) ,\varphi\rangle \to \langle \tdu(0) ,\varphi\rangle \quad a.s..
\end{align*}
By \eqref{eqn:wp-est-3} and Proposition \ref{p.L121801}-(3), one has 
\begin{align*}
    \sup_{j}\widetilde\E\|\tdu^{n_j}\|_{L^2(0,T;\mch)}^p \leq C \sup_{j}\widetilde\E\int_0^T\|\tdu^{n_j}\|_{L^2}^p dr\leq C(1+\widetilde\E\|\tdu_0\|_{L^2}^p)<\infty. 
\end{align*}
As $p>2$, by Vitali convergence theorem, 
\[\tdu^{n_j}\to \tdu\, \text{ in } L^2(\widetilde{\Omega};L^2(0,T;\mch)).\]
Therefore passing to a subsequence if necessary, we conclude that 
\begin{align}\label{eqn:122501}
    \|\tdu^{n_j}-\tdu\|_{L^2}\to 0,
\end{align}
for almost all $(t,\omega)\in[0,T]\times\widetilde{\Omega},$ so is $\langle\tdu^{n_j}(t),\varphi\rangle \to \langle\tdu(t),\varphi\rangle$. 
In view of the weak convergence \eqref{eq.L122503}, we have 
\begin{align*}
    \widetilde{\E}\int_0^T\mathbf{1}_A\langle (\nu\nabla \widetilde{u}^{n_j},D\nabla \widetilde{c_1}^{n_j}, D\nabla \widetilde{c_2}^{n_j}),  \nabla \varphi\rangle ds \to     \widetilde{\E}\int_0^T\mathbf{1}_A\langle (\nu\nabla \widetilde{u},D\nabla \widetilde{c_1}, D\nabla \widetilde{c_2}),  \nabla \varphi\rangle ds 
\end{align*}
for every measurable $A\subseteq[0,T]\times\widetilde{\Omega}$ and $\varphi\in \mcv$.  
Recall the notation
$\L_k = (\L_{\theta_k}, \L_{b_k}, \L_{b_k})$, we can write 
\begin{align*}
&\frac12\sum_{k}\widetilde{\E}\int_0^T\mathbf{1}_A\langle P_{n_j}\L_k\tdu^{n_j}, \L_k\varphi \rangle ds  -    \frac12\sum_{k}\widetilde{\E}\int_0^T\mathbf{1}_A\langle\L_k\tdu, \L_k\varphi \rangle ds
\\
= &\frac12\sum_{k}\widetilde{\E}\int_0^T\mathbf{1}_A\langle P_{n_j} \L_k\tdu^{n_j}- \L_k\tdu^{n_j}, \L_k\varphi \rangle ds + \frac12\sum_{k}\widetilde{\E}\int_0^T\mathbf{1}_A\langle  \L_k\tdu^{n_j} - \L_k\tdu, \L_k\varphi \rangle ds.
\end{align*}
Thanks to the property of the projection $P_{n_j}$ and the uniform bound for $\tdu^{n_j}$ in $L^2(\Omega; L^2(0,T; \mcv))$, as $b,\theta\in L_2(\Uc,L^\infty)$, for the first term in the above identity, we have
\begin{align*}
   &\frac12\sum_{k}\widetilde{\E}\int_0^T\mathbf{1}_A\langle P_{n_j} \L_k\tdu^{n_j}- \L_k\tdu^{n_j}, \L_k\varphi \rangle ds = \frac12\sum_{k}\widetilde{\E}\int_0^T\mathbf{1}_A\langle  \L_k\tdu^{n_j}, P_{n_j}\L_k\varphi - \L_k\varphi \rangle ds
   \\
   \leq & C \left(\sup_{n_j} \|\tdu^{n_j}\|_{L^2(\Omega; L^2(0,T; \mcv))}\right) \|P_{n_j}\varphi - \varphi\|_{H^1} \to 0.
\end{align*}
Thanks to the weak convergence \eqref{eq.L122503}, the second term also converges to $0$. Thus,
\begin{align*}
    \frac12\sum_{k}\widetilde{\E}\int_0^T\mathbf{1}_A\langle P_{n_j}\L_k\tdu^{n_j}, \L_k\varphi \rangle ds \to \frac12\sum_{k}\widetilde{\E}\int_0^T\mathbf{1}_A\langle\L_k\tdu, \L_k\varphi \rangle ds.
\end{align*}

Next, as $b,\theta\in L_2(\Uc,L^\infty)$ and thanks to \eqref{eqn:122501} and the uniform bound for $\tdu^{n_j}$ in $L^2(\Omega; L^\infty(0,T; \mch))$, we have
\begin{align*}
    \sum_k|\langle \tdu^{n_j}-\tdu, \L_k\varphi\rangle|^2\leq C\|\nabla\varphi\|_{L^2}^2\|\tdu^{n_j}-\tdu\|_{\mch}^2 \to 0 
\end{align*}
and
\begin{align*}
    \sum_k|\langle \tdu^{n_j}, \L_k(P_{n_j}\varphi-\varphi)\rangle|^2 \leq C\|P_{n_j}\varphi-\varphi\|_{L^2} \to 0
\end{align*}
for almost every $(t,\omega)\in [0,T]\times \widetilde\Omega$.
Hence by denoting 
$\L = (\theta\cdot\nabla, b\cdot\nabla, b\cdot\nabla)$,
we have 
\begin{align*}
    \langle \tdu^{n_j}, \L P_{n_j}\varphi\rangle\to \langle \tdu, \L\varphi\rangle
\end{align*}
in probability in $L^2(0,T;L_2(\Uc,\mathbb{R}))$. In addition, one has $\widetilde{\mathbb{W}}^{n_j}\to \widetilde{\mathbb{W}}$ almost surely in $C\left(\left[0, T\right], \Uc_0\right)$ by Proposition \ref{p.L121801}. As a result, Lemma 2.1 from \cite{{debussche2011local}} implies that 
\begin{align*}
    \int_0^t \langle \tdu^{n_j}, \L\varphi\rangle d\widetilde{\mathbb{W}}_s^{n_j}\to     \int_0^t \langle \tdu, \L\varphi\rangle d\widetilde{\mathbb{W}}_s
\end{align*}
in probability in $L^2(0,T;\mathbb{R})$. An application of the BDG inequality and the Vitali convergence theorem show that the convergence occurs in $L^2(\widetilde{\Omega};L^2(0,T;\mathbb{R}))$. Therefore by passing to a subsequence if necessary the convergence in fact also holds 
for almost every $(t,\omega)\in[0,T]\times\widetilde{\Omega}$. 
 
{\it Step 2: Convergence of the nonlinear terms.} Fix any test function $\phi\in H^1$. 
First we prove that for $i=1,2$,
 \begin{align}\label{eq.L121901}
     \int_0^t\widetilde{\tau}_{\eta}^{n_j} \langle  \widetilde{c_i}^{n_j} \nabla \widetilde{\Phi}^{n_j}, \nabla\phi\rangle ds\to      \int_0^t\widetilde{\tau}_{\eta} \langle  \widetilde{c_i} \nabla \widetilde{\Phi}, \nabla\phi\rangle ds 
 \end{align}
for almost every $(t,\omega)\in[0,T]\times\widetilde{\Omega}$. We present the proof for $i=1$, and the case of $i=2$ is identical. In view of \eqref{eqn:rho}, one has
\begin{align*}
\begin{split}
    &\left|\int_0^t\widetilde{\tau}_{\eta}^{n_j} \langle  \widetilde{c_1}^{n_j} \nabla \widetilde{\Phi}^{n_j}, \nabla\phi\rangle ds - \int_0^t\widetilde{\tau}_{\eta} \langle  \widetilde{c_1} \nabla \widetilde{\Phi}, \nabla\phi\rangle ds \right|\\
    &\leq \int_0^t\left|\widetilde{\tau}_{\eta}^{n_j}-\widetilde{\tau}_{\eta}\right| |\langle  \widetilde{c_1}^{n_j} \nabla \widetilde{\Phi}^{n_j}, \nabla\phi\rangle|ds+\int_0^t\left| \langle  (\widetilde{c_1}^{n_j}-\widetilde{c_1}) \nabla \widetilde{\Phi}^{n_j}, \nabla\phi\rangle\right|ds+\int_0^t\left|\langle  \widetilde{c_1} \nabla (\widetilde{\Phi}^{n_j}-\widetilde{\Phi}), \nabla\phi\rangle\right|ds \\
    &:=I_1+I_2+I_3.
\end{split}
\end{align*}
The Lipschitz continuity for $\tau_{\eta}$ yields 
\[\left|\widetilde{\tau}_{\eta}^{n_j}-\widetilde{\tau}_{\eta}\right| = \left|\tau_{\eta}(\|\widetilde{c_1}^{n_j}\|_{L^2}+ \|\widetilde{c_2}^{n_j}\|_{L^2}) - \tau_{\eta}(\|\widetilde{c_1}\|_{L^2}+ \|\widetilde{c_2}\|_{L^2})\right|\leq C_\eta\|\tdu^{n_j}-\tdu\|_{\mch}.\]
Therefore by the Sobolev inequality, the H\"older inequality, the elliptic estimate, and \eqref{eq.L121801} with Proposition \ref{p.L121801}-(3), for almost all $(t,\omega) \in [0,T]\times \widetilde\Omega$ we have
\begin{align*}
    I_1&\leq \int_0^T\|\tdu^{n_j}-\tdu\|_{\mch}\|\widetilde{c_1}^{n_j}\|_{L^4}\|\nabla \widetilde{\Phi}^{n_j}\|_{L^4}\|\nabla \phi\|_{L^2}ds\\
    &\leq C \int_0^T\|\tdu^{n_j}-\tdu\|_{\mch}\|\widetilde{c_1}^{n_j}\|_{H^1}\| \widetilde{\rho}^{n_j}\|_{L^2}\|\nabla \phi\|_{L^2}ds\\
    &\leq C\sup_{s\in[0,T]}\|\tdu^{n_j}(s)\|_{L^2}\|\nabla \phi\|_{L^2}\int_0^T\|\tdu^{n_j}-\tdu\|_{\mch}\|\tdu^{n_j}\|_{\mcv}ds\\
    &\leq C (1+\|\tdu_0^{n_j}\|_{L^2})\|\nabla \phi\|_{L^2}\|\tdu^{n_j}-\tdu\|_{L^2(0,T;\mch)}\|\tdu^{n_j}\|_{L^2(0,T;\mcv)}\\
    &\leq C (1+\|\tdu_0^{n_j}\|_{L^2}^{2})\|\nabla \phi\|_{L^2}\|\tdu^{n_j}-\tdu\|_{L^2(0,T;\mch)}.
\end{align*}
Similarly, by the interpolation inequality we have 
\begin{align*}
    I_2&\leq \int_0^t\|\widetilde{c_1}^{n_j}-\widetilde{c_1}\|_{L^4}\|\nabla \widetilde{\Phi}^{n_j}\|_{L^4}\|\nabla \phi\|_{L^2}ds\\
    &\leq C\int_0^T\|\widetilde{c_1}^{n_j}-\widetilde{c_1}\|_{L^2}^{\frac12}\|\widetilde{c_1}^{n_j}-\widetilde{c_1}\|_{H^1}^{\frac12}\|\widetilde{\rho}^{n_j}\|_{L^2}\|\nabla \phi\|_{L^2}ds\\
    &\leq C (1+\|\tdu_0^{n_j}\|_{L^2})\|\nabla \phi\|_{L^2}\|\tdu^{n_j}-\tdu\|_{L^2(0,T;\mch)}^{\frac12}\|\tdu^{n_j}-\tdu\|_{L^2(0,T;\mcv)}^{\frac12},
\end{align*}
and 
\begin{align*}
    I_3&\leq \int_0^t\|\nabla\widetilde{\Phi}^{n_j}-\nabla \widetilde{\Phi}\|_{L^4}\|\widetilde{c_1}\|_{L^4}\|\nabla \phi\|_{L^2}ds\\
    &\leq C\int_0^T\|\widetilde{\rho}^{n_j}-\widetilde{\rho}\|_{L^2}\|\widetilde{c_1}\|_{H^1}\|\nabla \phi\|_{L^2}ds\\
    &\leq C \|\nabla \phi\|_{L^2}\|\tdu^{n_j}-\tdu\|_{L^2(0,T;\mch)}\|\widetilde{c_1}\|_{L^2(0,T;H^1)}\\
    &\leq C (1+\|\tdu_0^{n_j}\|_{L^2})\|\nabla \phi\|_{L^2}\|\tdu^{n_j}-\tdu\|_{L^2(0,T;\mch)}.
\end{align*}
The convergence of \eqref{eq.L121901} then follows from estimates for $I_1,I_2,I_3$ and the almost sure convergence of $\tdu^{n_j}$ to $\tdu$ in $L^2(0,T;\mch)$, as well as \eqref{eq.L121902}. 

In a similar fashion, we prove that for $i=1,2$,
\begin{align}\label{eq.L121903}
    \int_0^t\langle \widetilde{u}^{n_j}\cdot\nabla \phi, \widetilde{c_i}^{n_j}\rangle ds\to    \int_0^t\langle \widetilde{u}\cdot\nabla \phi, \widetilde{c_i}\rangle ds
\end{align}
for almost every $(t,\omega)\in[0,T]\times\widetilde{\Omega}$. Note that 
\begin{align*}
    &\left|\int_0^t\langle \widetilde{u}^{n_j}\cdot\nabla \phi, \widetilde{c_i}^{n_j}\rangle ds-\int_0^t\langle \widetilde{u}\cdot\nabla \phi, \widetilde{c_i}\rangle ds\right| \\
    &\leq \int_0^t\left|\langle (\widetilde{u}^{n_j}-\widetilde{u})\cdot\nabla \phi, \widetilde{c_i}^{n_j}\rangle\right|ds + \int_0^t\left|\langle \widetilde{u}\cdot\nabla \phi, \widetilde{c_i}^{n_j}-\widetilde{c_i}\rangle\right|ds\\
    &\leq  \int_0^t\|\widetilde{u}^{n_j}-\widetilde{u}\|_{L^4}\|\nabla \phi\|_{L^2}\|\widetilde{c_i}^{n_j}\|_{L^4}ds + \int_0^t\|\widetilde{u}\|_{L^4}\|\nabla \phi\|_{L^2}\|\widetilde{c_i}^{n_j}-\widetilde{c_i}\|_{L^4}ds\\
    &\leq C\|\nabla \phi\|_{L^2}\|\tdu^{n_j}-\tdu\|_{L^2(0,T;\mch)}^{\frac12}\|\tdu^{n_j}-\tdu\|_{L^2(0,T;\mcv)}^{\frac12}\left( \|\tdu^{n_j}\|_{L^2(0,T;\mcv)}+ \|\tdu\|_{L^2(0,T;\mcv)}\right).
\end{align*}
Hence the convergence \eqref{eq.L121903} follows again from the almost sure convergence of $\tdu^{n_j}$ to $\tdu$ in $L^2(0,T;\mch)$, estimate \eqref{eq.L121801} and \eqref{eq.L121902}. 

By the H\"older inequality and the Ladyzhenskaya inequality, and Young's inequality for products, one obtains 
\begin{align*}
    \left|\int_0^t\widetilde{\tau}_{\eta}^{n_j} \langle  \widetilde{c_i}^{n_j} \nabla \widetilde{\Phi}^{n_j}, \nabla\phi\rangle ds\right| &\leq  \|\nabla\phi\|_{L^2}\int_0^T\|\widetilde{c_i}^{n_j}\|_{L^4} \|\nabla \widetilde{\Phi}^{n_j}\|_{L^4} dt\\
    &\leq C\int_0^T\|\widetilde{c_i}^{n_j}\|_{L^2}^{1/2}\|\nabla\widetilde{c_i}^{n_j}\|_{L^2}^{1/2}\|\nabla \widetilde{\Phi}^{n_j}\|_{L^2}^{1/2}\|\widetilde{\rho}^{n_j}\|_{L^2}^{1/2}dt\\
    &\leq C\int_0^T\left(\|\tdu^{n_j}\|_{L^2}^2+\|\nabla \tdu^{n_j}\|_{L^2}^2\right)dt\\
    &\leq C(1+\|U(0)\|_{L^2}^2),
\end{align*}
which has a finite expectation. Hence Lebesgue's dominated convergence theorem implies that the convergence \eqref{eq.L121901} also holds in $L^1([0,T]\times\widetilde{\Omega})$. Similarly the convergence \eqref{eq.L121903} happens in $L^1([0,T]\times\widetilde{\Omega})$. 

We also have the convergence for the two nonlinear terms in the equations of $u$ for any test function $\phi\in V$: 
\begin{align*}
\begin{split}
    \int_0^t\widetilde{\tau}_{\eta}^{n_j} \langle  \widetilde{\rho}^{n_j} \nabla \widetilde{\Phi}^{n_j}, \phi\rangle ds&\to      \int_0^t\widetilde{\tau}_{\eta} \langle  \widetilde{\rho} \nabla \widetilde{\Phi}, \phi\rangle ds, \\
    \int_0^t\langle \widetilde{u}^{n_j}\cdot\nabla \phi, \widetilde{u}^{n_j}\rangle ds&\to    \int_0^t\langle \widetilde{u}\cdot\nabla \phi, \widetilde{u}\rangle ds,  
\end{split}
\end{align*}
in $L^1([0,T]\times\widetilde{\Omega})$. The proof is omitted since it is similar to the proof of \eqref{eq.L121901} and \eqref{eq.L121903}. 

{\it Step 3: Identify the limit $\tdu$ as a martingale solution.} For any $U=(u,c_1,c_2)$, any cylindrical Wiener process $\mathbb W$, and any $v=(\varphi,\phi_1,\phi_2)\in\mcv$, we denote

\begin{align*}
    &L(t,U,v)=-\int_0^t\langle (\nu\nabla u, D\nabla c_1, D\nabla c_2), \nabla v\rangle ds - \frac12\sum_{k}\int_0^t\langle \L_kU,\L_k v\rangle ds \\
    &S(t,U,v,\mathbb W)=-\int_0^t\langle U, \L  v\rangle d\mathbb W_s, \text{ where } \L = (\theta\cdot\nabla, b\cdot\nabla, b\cdot\nabla),\\
    &N(t,U,v) = \int_0^t \sum_{i=1}^2(\langle u\cdot\nabla \phi_i,c_i\rangle + (-1)^{i}\tau_{\eta}\langle Dc_i\nabla\Phi,\nabla\phi_i \rangle)ds 
   +\int_0^t\left(\langle u\cdot\nabla \varphi, u\rangle + \langle\rho\nabla\Phi,\varphi \rangle  \right)ds,  
\end{align*}
keeping in mind that $\tau_{\eta}=\tau_{\eta}(\|c_1\|_{L^2}+\|c_2\|_{L^2}),\rho=c_1-c_2$ and $-\Delta\Phi=\rho$. Proposition \ref{p.L121801}-(3) implies that $\tdu^{n_j},\widetilde{\mathbb W}^{n_j}$ satisfy
\begin{align*}
    \widetilde{\E}\int_0^T\mathbf{1}_A\langle\tdu^{n_j}(t),P_{n_j}v\rangle dt=\widetilde{\E}\int_0^T\mathbf{1}_A\langle\tdu^{n_j}_0,P_{n_j}v\rangle dt&+ \widetilde{\E}\int_0^T\mathbf{1}_A\left(L(t,\tdu^{n_j},P_{n_j}v) + N(t,\tdu^{n_j},P_{n_j}v)\right)dt \\
    &\qquad  +\widetilde{\E}\int_0^T \mathbf{1}_A S(t,\tdu^{n_j},P_{n_j}v,\widetilde{\mathbb W}^{n_j})dt,
\end{align*}
for every measurable $A\subseteq [0,T]\times\widetilde{\Omega}$. By taking limit $n_j\to\infty$, it follows from {\it Step 2 } that the limit $(\tdu_0,\tdu,\widetilde{\mathbb W})$ satisfies 
\begin{align*}
    \widetilde{\E}\int_0^T\mathbf{1}_A\langle\tdu(t),v\rangle dt=\widetilde{\E}\int_0^T\mathbf{1}_A\langle\tdu_0,v\rangle dt&+ \widetilde{\E}\int_0^T\mathbf{1}_A\left( L(t,\tdu,v) + N(t,\tdu,v)\right)dt \\
    &\qquad  +\widetilde{\E}\int_0^T \mathbf{1}_A S(t,\tdu,v,\widetilde{\mathbb W})dt,
\end{align*}
for every measurable set $A\subseteq [0,T]\times\widetilde{\Omega}$ and $v\in\mcv$. 

Therefore the limit $(\tdu_0,\tdu,\widetilde{\mathbb W})$ indeed satisfies the cut-off system \eqref{sys:modified} in the sense of \eqref{eq:solution.def}. 

It remains to show that $\tdu\in L^2(\widetilde{\Omega}; C([0,T]; \mch))$, and this follows from a similar argument as in Section 7.3 of \cite{debussche2011local}. Thus $\tdu$ is a martingale solution to the cut-off system \eqref{sys:modified}. 
\end{proof}
% After proving these two propositions, by applying suitable stopping time we can get local martingale solutions for \eqref{e.w04091}.

\subsection{Pathwise uniqueness of the cut-off system}\label{s.L122202}
In this section we prove the pathwise uniqueness of the cut-off system \eqref{sys:modified}.

\begin{proposition}
\label{p.w05241}
Let $u_0 \in L^p\left( \Omega,H \right)$ and $c_i(0)\in L^q\left( \Omega, L^2 \right)$ be $\Fc_0$ measurable with $p \geq 2$ and $q\geq 4$. Let $\Sc = \left(\Omega, \Fc, \Fct, \Pb \right)$ and $W$ be fixed. Assume that $\left(\Sc, W, u^{(1)}, c_1^{(1)}, c_2^{(1)}\right)$ and $\left(\Sc, W, u^{(2)}, c_1^{(2)}, c_2^{(2)}\right)$ are two global martingale solutions of the modified equation \eqref{sys:modified}. Denote by $\Omega_0 = \left\lbrace u_1(0) = u_2(0) \right\rbrace \wedge \left\lbrace c^{(1)}_1(0) = c^{(2)}_1(0) \right\rbrace \wedge \left\lbrace c^{(1)}_2(0) = c^{(2)}_2(0) \right\rbrace \subseteq \Omega$. Then,
\begin{equation}\label{eqn:uniqueness-result}
	\mathbb P \left( \left\lbrace \mathds{1}_{\Omega_0}\left(|u^{(1)}(t) - u^{(2)}(t)| + |c_1^{(1)}(t)-c_1^{(2)}(t)| +  |c_2^{(1)}(t)-c_2^{(2)}(t)| \right) = 0 \ \text{for all} \ t \geq 0 \right\rbrace \right) = 1.
\end{equation}
\end{proposition}

\begin{proof}
    Let $u=u^{(1)}-u^{(2)}$, $c_1=c_1^{(1)}-c_1^{(2)}$, $c_2=c_2^{(1)}-c_2^{(2)}$, $\Phi=\Phi^{(1)}-\Phi^{(2)}$, $\rho=\rho^{(1)}-\rho^{(2)}$.
    Denote by the stopping times
    \[
     \tau^{(1)}_\eta = \tau_\eta(\|c_1^{(1)}\|_{L^2} + \|c_2^{(1)}\|_{L^2}), \quad \tau^{(2)}_\eta = \tau_\eta(\|c_1^{(2)}\|_{L^2} + \|c_2^{(2)}\|_{L^2}).
    \]
    It follows that the equations for $u$, $c_1$ and $c_2$ are
    \begin{align}\label{sys:modified-difference}
    \begin{split}
    dc_1&=\Big(D\Delta c_1 - u\cdot \nabla c^{(1)}_1 - u^{(2)} \cdot \nabla c_1 + D (\tau_\eta^{(1)} - \tau_\eta^{(2)})\nabla \cdot (c_1^{(1)} \nabla \Phi^{(1)}) + D \tau_\eta^{(2)} \nabla \cdot (c_1\nabla \Phi^{(1)}) 
    \\
    &\hspace{2cm}+ D \tau_\eta^{(2)} \nabla\cdot (c_1^{(2)} \nabla\Phi)  + \frac{1}{2} \sum\limits_{k} \L_{b_k}^2c_1 \Big)dt 
    + \sum\limits_{k}\L_{b_k}c_1d{W}_t^k, 
    \\
    dc_2&=\Big(D\Delta c_2 - u\cdot \nabla c^{(1)}_2 - u^{(2)} \cdot \nabla c_2 - D (\tau_\eta^{(1)} - \tau_\eta^{(2)})\nabla \cdot (c_2^{(1)} \nabla \Phi^{(1)}) - D \tau_\eta^{(2)} \nabla \cdot (c_2\nabla \Phi^{(1)}) 
    \\
    &\hspace{2cm}- D \tau_\eta^{(2)} \nabla\cdot (c_2^{(2)} \nabla\Phi)  + \frac{1}{2} \sum\limits_{k} \L_{b_k}^2c_2 \Big)dt 
    + \sum\limits_{k}\L_{b_k}c_2 d{W}_t^k, 
    \\
    du&=\Big(\nu \Delta u- \Pi\Big(  u\cdot \nabla u^{(1)} + u^{(2)}\cdot \nabla u + (\tau_{\eta}^{(1)}-\tau_\eta^{(2)})\rho^{(1)} \nabla \Phi^{(1)} + \tau_\eta^{(2)} \rho \nabla\Phi^{(1)} + \tau_\eta^{(2)}\rho^{(2)}\nabla\Phi\Big)  
    \\
    &\hspace{2cm}+ \frac{1}{2} \sum\limits_{k}  \L_{\theta_k}^2u \Big)dt +  \sum\limits_{k} \L_{\theta_k}u d{W}_t^k, 
    \end{split}
\end{align}
with $-\Delta \Phi = \rho = c_1 - c_2$ and $u(0)=u^{(1)}(0)-u^{(2)}(0)$, $c_1(0)=c_1^{(1)}(0)-c_1^{(2)}(0)$, $c_2(0)=c_2^{(1)}(0)-c_2^{(2)}(0)$. Let $\xi^n$ be the stopping time defined as 
\begin{align*}
    \xi^n = \inf \left\{ t\geq 0: \int_0^t \Big(1+ \|\nabla u^{(1)}\|_{L^2}^2 + \sum\limits_{i=1}^2\|\nabla c_i^{(1)}\|_{L^2}^2  + \sum\limits_{i=1}^2\sum\limits_{j=1}^2 (\|c_i^{(j)}\|_{L^2}^4 + \|c_i^{(j)}\|_{L^2}^2 \|\nabla c_i^{(j)}\|_{L^2}^2) \Big) ds \geq n \right\}.
\end{align*}
By virtue of \eqref{eqn:wp-est-1} and \eqref{eqn:wp-est-2}, we deduce that $\lim\limits_{n\to\infty} \xi^n = \infty$ $\mathbb P$-a.s. as long as $u_0 \in L^2\left( \Omega, H \right)$ and $c_i(0)\in L^4\left( \Omega, L^2 \right)$. Therefore, it is sufficient to prove that
\begin{align*}
    \sup\limits_{s\in[0,\xi^n\wedge t]} \left( \|u\|_{L^2}^2 +  \|c_1\|_{L^2}^2 +  \|c_2\|_{L^2}^2 \right) = 0
\end{align*}
for almost all $\omega \in \Omega_0$ and for all $t\geq 0$.

Fix $n\in\mathbb N$ and $t\geq 0$, we denote by the stopping time $\xi=\xi^n\wedge t$. We now calculate by It\^o's formula that 
\begin{equation}\label{est:uniqueness-1}
    \begin{split}
        & \sup\limits_{s\in[0,\xi]} \left( \|u\|_{L^2}^2 +  \|c_1\|_{L^2}^2 +  \|c_2\|_{L^2}^2 \right) + 2 \int_{0}^{\xi} \left(\nu\|\nabla u\|_{L^2}^2 + D\|\nabla c_1\|_{L^2}^2 + D\|\nabla c_2\|_{L^2}^2\right)ds
        \\
        \leq & \|u(0)\|_{L^2}^2 +  \|c_1(0)\|_{L^2}^2 +  \|c_2(0)\|_{L^2}^2 
        \\
        &+ 2\int_{0}^{\xi} \left| \left\langle u\cdot \nabla c_1^{(1)} , c_1\right\rangle \right| ds + 2\int_{0}^{\xi} \left| \left\langle u\cdot \nabla c_2^{(1)} , c_2\right\rangle \right| ds + 2\int_{0}^{\xi} \left| \left\langle u\cdot \nabla u^{(1)} , u\right\rangle \right| ds 
        \\
        &+ 2D \int_{0}^{\xi} \Big(\left| \left\langle (\tau_\eta^{(1)} -\tau_\eta^{(2)})c_1^{(1)} \nabla\Phi^{(1)} , \nabla c_1\right\rangle \right| + \left| \left\langle \tau_\eta^{(2)} c_1 \nabla\Phi^{(1)} , \nabla c_1\right\rangle \right| + \left| \left\langle \tau_\eta^{(2)} c_1^{(2)} \nabla\Phi , \nabla c_1\right\rangle \right| \Big)ds
        \\
        &+ 2D \int_{0}^{\xi} \Big(\left| \left\langle (\tau_\eta^{(1)} -\tau_\eta^{(2)})c_2^{(1)} \nabla\Phi^{(1)} , \nabla c_2\right\rangle \right| + \left| \left\langle \tau_\eta^{(2)} c_2 \nabla\Phi^{(1)} , \nabla c_2\right\rangle \right| + \left| \left\langle \tau_\eta^{(2)} c_2^{(2)} \nabla\Phi , \nabla c_2\right\rangle \right| \Big) ds
        \\
        &+ 2 \int_{0}^{\xi} \Big(\left| \left\langle (\tau_\eta^{(1)} -\tau_\eta^{(2)})\rho^{(1)} \nabla\Phi^{(1)} , u\right\rangle \right| + \left| \left\langle \tau_\eta^{(2)} \rho \nabla\Phi^{(1)} , u\right\rangle \right| + \left| \left\langle \tau_\eta^{(2)} \rho^{(2)} \nabla\Phi , u\right\rangle \right| \Big) ds,
    \end{split}
\end{equation}
where we have used the cancellation 
\[
  \langle \L_{b_k}^2c_i, c_i \rangle + \|\L_{b_k} c_i\|_{L^2}^2 = 0, \quad  \langle \L_{\theta_k}^2u, u \rangle + \|\L_{\theta_k} u\|_{L^2}^2 = 0,
\]
and the $dW$ terms vanish due to the divergence free condition of $b_k$ and $\theta_k$. Now we estimate terms in \eqref{est:uniqueness-1}. By the H\"older inequality, Young's inequality, and the interpolation inequality, we have
\begin{equation}\label{est:uniqueness-21}
\begin{split}
    &2 \int_{0}^{\xi} \left| \left\langle u\cdot \nabla c_1^{(1)} , c_1\right\rangle \right| ds + 2 \int_{0}^{\xi} \left| \left\langle u\cdot \nabla c_2^{(1)} , c_2\right\rangle \right| ds + 2 \int_{0}^{\xi} \left| \left\langle u\cdot \nabla u^{(1)} , u\right\rangle \right| ds
    \\
    \leq &\frac13  \int_{0}^{\xi} \left(\nu\|\nabla u\|_{L^2}^2 + D\|\nabla c_1\|_{L^2}^2 + D\|\nabla c_2\|_{L^2}^2\right)ds 
    \\
    &+ C  \int_{0}^{\xi} \left(1+ \|\nabla u^{(1)}\|_{L^2}^2 + \|\nabla c_1^{(1)}\|_{L^2}^2 + \|\nabla c_2^{(1)}\|_{L^2}^2 \right) \left( \|u\|_{L^2}^2 +  \|c_1\|_{L^2}^2 +  \|c_2\|_{L^2}^2\right) ds.
\end{split}
\end{equation}
Since $\tau_{\eta}$ is  
Lipschitz continuous by definition \eqref{eqn:rho}, using the elliptic bound it follows that
\begin{equation}\label{est:uniqueness-31}
\begin{split}
    &2D \int_{0}^{\xi} \Big(\left| \left\langle (\tau_\eta^{(1)} -\tau_\eta^{(2)})c_1^{(1)} \nabla\Phi^{(1)} , \nabla c_1\right\rangle \right| + \left| \left\langle \tau_\eta^{(2)} c_1 \nabla\Phi^{(1)} , \nabla c_1\right\rangle \right| + \left| \left\langle \tau_\eta^{(2)} c_1^{(2)} \nabla\Phi , \nabla c_1\right\rangle \right| \Big)ds
    \\
    &+2D \int_{0}^{\xi} \Big(\left| \left\langle (\tau_\eta^{(1)} -\tau_\eta^{(2)})c_2^{(1)} \nabla\Phi^{(1)} , \nabla c_2\right\rangle \right| + \left| \left\langle \tau_\eta^{(2)} c_2 \nabla\Phi^{(1)} , \nabla c_2\right\rangle \right| + \left| \left\langle \tau_\eta^{(2)} c_2^{(2)} \nabla\Phi , \nabla c_2\right\rangle \right| \Big) ds
    \\
    \leq &\frac13  \int_{0}^{\xi} \left( D\|\nabla c_1\|_{L^2}^2 + D\|\nabla c_2\|_{L^2}^2\right)ds 
    \\
    &+ C \int_{0}^{\xi} \left(1+   \sum\limits_{i=1}^2\sum\limits_{j=1}^2 (\|c_i^{(j)}\|_{L^2}^4 + \|c_i^{(j)}\|_{L^2}^2 \|\nabla c_i^{(j)}\|_{L^2}^2)  \right) \left( \|c_1\|_{L^2}^2 +  \|c_2\|_{L^2}^2\right) ds,
\end{split}
\end{equation}
and
\begin{equation}\label{est:uniqueness-41}
\begin{split}
    &2 \int_{0}^{\xi} \Big(\left| \left\langle (\tau_\eta^{(1)} -\tau_\eta^{(2)})\rho^{(1)} \nabla\Phi^{(1)} , u\right\rangle \right| + \left| \left\langle \tau_\eta^{(2)} \rho \nabla\Phi^{(1)} , u\right\rangle \right| + \left| \left\langle \tau_\eta^{(2)} \rho^{(2)} \nabla\Phi , u\right\rangle \right| \Big) ds
    \\
    \leq &\frac13  \int_{0}^{\xi} \nu\|\nabla u\|_{L^2}^2 ds + C \int_{0}^{\xi} \left(1+   \sum\limits_{i=1}^2\sum\limits_{j=1}^2 \|c_i^{(j)}\|_{L^2}^4   \right) \left( \|c_1\|_{L^2}^2 +  \|c_2\|_{L^2}^2\right) ds.
\end{split}
\end{equation}
Combining the estimates \eqref{est:uniqueness-21}-\eqref{est:uniqueness-41}, one has
\begin{equation}\label{est:uniqueness-2}
    \begin{split}
        & \sup\limits_{s\in[0,\xi]} \left( \|u\|_{L^2}^2 +  \|c_1\|_{L^2}^2 +  \|c_2\|_{L^2}^2 \right) +  \int_{0}^{\xi} \left(\nu\|\nabla u\|_{L^2}^2 + D\|\nabla c_1\|_{L^2}^2 + D\|\nabla c_2\|_{L^2}^2\right)ds
        \\
        \leq &  \left[\|u(0)\|_{L^2}^2 +  \|c_1(0)\|_{L^2}^2 +  \|c_2(0)\|_{L^2}^2 \right]
        \\
        &+  C  \int_{0}^{\xi} \left(1+ \|\nabla u^{(1)}\|_{L^2}^2 + \sum\limits_{i=1}^2\|\nabla c_i^{(1)}\|_{L^2}^2  + \sum\limits_{i=1}^2\sum\limits_{j=1}^2 (\|c_i^{(j)}\|_{L^2}^4 + \|c_i^{(j)}\|_{L^2}^2 \|\nabla c_i^{(j)}\|_{L^2}^2) \right) 
        \\
        &\hspace{3cm}\times\left( \|u\|_{L^2}^2 +  \|c_1\|_{L^2}^2 +  \|c_2\|_{L^2}^2\right) ds.
    \end{split}
\end{equation}
Applying the Gronwall inequality, we conclude that for any $t\geq 0$,
\begin{align*}
    \sup\limits_{s\in[0,\xi^n\wedge t]} \left( \|u\|_{L^2}^2 +  \|c_1\|_{L^2}^2 +  \|c_2\|_{L^2}^2 \right) = 0
\end{align*}
for almost all $\omega \in \Omega_0$. Consequently, \eqref{eqn:uniqueness-result} follows.
\end{proof}

\section{Non-negativity and global well-posedness}\label{sec:nonneg}
The global pathwise solution of the cut-off system from Theorem \ref{t.L122201} gives a local pathwise solution of the original NPNS system \eqref{e.L112403} by considering a stopping time before which the cut-off function $\tau_\eta$ is not activated, i.e., $\tau_\eta=1$. In this section we extend this local solution to a global one by establishing an {\it a priori} estimate. A crucial step toward the estimate is to show the non-negativity of the ionic concentrations.

For any real number $X$ we denote 
\[X^{+}:=\max \{X, 0\}, \quad X^{-}:=\max \{-X, 0\}.\]
\begin{lemma}\label{l.L110201}
Assume that the initial concentration $c_1(0), c_2(0)\in L^2(\Omega,L^2)$ are non-negative almost surely, and $(u,c_1, c_2)$ is the corresponding local pathwise solution of system \eqref{e.L112403}. Then $c_1(t), c_2(t)$ remain non-negative almost surely for any $t>0$ as long as the solution exists. 
\end{lemma}
\begin{proof}
Consider the auxiliary problem 
\begin{align}\label{e.L122401}
    \begin{split}
    dc_1&=\left(D\Delta c_1 - u\cdot \nabla c_1  + D \nabla \cdot (c^+_1 \nabla \Phi) + \frac{1}{2} \sum\limits_{k} \L_{b_k}^2c_1 \right)dt 
    + \sum\limits_{k}\L_{b_k}c_1d{W}_t^k, 
    \\
    dc_2&=\left(D\Delta c_2 -u\cdot \nabla c_2  - D \nabla \cdot (c^+_2 \nabla \Phi) +  \frac{1}{2} \sum\limits_{k}\L_{b_k}^2c_2  \right)dt
    + \sum\limits_{k}\L_{b_k}c_2 d{W}_t^k,
    \\
    du&=\left(\nu \Delta u- \Pi\Big(  u\cdot \nabla u + \rho \nabla \Phi\Big)  + \frac{1}{2} \sum\limits_{k}  \L_{\theta_k}^2u \right)dt+  \sum\limits_{k} \L_{\theta_k}u d{W}_t^k, 
    \\
    \nabla\cdot u &=0,
    \\
    -\Delta \Phi &= \rho = c_1 - c_2,
    \end{split}
\end{align}

If $c_1, c_2$ of the solution $(u,c_1,c_2)$ to this auxiliary equation are non-negative, then $(u,c_1, c_2)$ is a solution of \eqref{e.L112403}. In particular, this implies that the solution $c_1, c_2$ to system \eqref{e.L112403} are non-negative. We only prove that $c_1$ remains non-negative below, and the case for $c_2$ is the same.

By It\^o's formula (for example, see \cite{kry2013}) we have 
\begin{align}\label{e.L110201}
\begin{split}
    \|c_1^-(t)\|_{L^2}^2 &= \|c_1^-(0)\|_{L^2}^2 + \int_0^t\int_{\mathbb{T}^2}\mathbf{1}_{\{c_1(s)<0\}}\sum_{k}(\L_{b_k} c_1(s))^2dxds -
    2\sum_{k}\int_0^t \int_{\mathbb{T}^2}c_1^-(s)L_{b_k} c_1(s)dxdW_s^k \\
    &\qquad - 2\int_0^t\left\langle c_1^-, D \Delta c_1-u \cdot \nabla c_1+D \nabla \cdot\left(c_1^{+} \nabla \Phi\right)+\frac{1}{2} \sum_{k}\mathcal{L}_{b_k}^2 c_1\right\rangle ds.
\end{split}
\end{align}
We have 
$\|c_1^-(0)\|_{L^2}^2 = 0$ since $c_1(0)$ is non-negative. Since $u$ is divergence free, and as one of $c_1^+$ and $c_1^{-}$ must be zero, one has 
\[\langle u\cdot\nabla c_1, c_1^-\rangle = -\langle u\cdot\nabla c_1^-, c_1^-\rangle = 0, \quad \langle D \nabla \cdot\left(c_1^{+} \nabla \Phi\right) , c_1^-\rangle =0.
\]
Therefore,
\begin{align*}
   -\left\langle c_1^-, D \Delta c_1-u \cdot \nabla c_1+D \nabla \cdot\left(c_1^{+} \nabla \Phi\right)+\frac{1}{2} \sum_{k} \mathcal{L}_{b_k}^2  c_1\right\rangle = -D\|\nabla c_1^-\|_{L^2}^2 + \frac12\sum_{k}\langle c_1^-, \mathcal{L}_{b_k}^2 c_1^-\rangle. 
\end{align*}
Therefore by denoting
\[M_t = 2\sum_{k}\int_0^t \int_{\mathbb{T}^2}c_1^-(s)\L_{b_k}c_1^-(s)dxdW_s^k,\]
equation \eqref{e.L110201} becomes 
\begin{align*}
    \left\|c_1^{-}(t)\right\|_{L^2}^2 = \sum_{k}\int_0^t\|\L_{b_k} c_1^-\|_{L^2}^2ds + M_t - 2D\int_0^t\|\nabla c_1^-\|_{L^2}^2ds + \sum_{k}\int_0^t\langle c_1^-, \mathcal{L}_{b_k}^2 c_1^-\rangle ds
\end{align*}
Since $b_k$ is divergence free, the dual of $\mathcal{L}_{b_k}$ is $-\mathcal{L}_{b_k}$. Hence we have the cancellation 
\begin{align*}
    \|\L_{b_k} c_1^-\|_{L^2}^2 +  \langle c_1^-, \mathcal{L}_{b_k}^2 c_1^-\rangle = 0. 
\end{align*}
As a result, the martingale $M_t$ starting from zero satisfies 
\begin{align*}
    0\leq \left\|c_1^{-}(t)\right\|_{L^2}^2 + 2D\int_0^t\|\nabla c_1^-\|_{L^2}^2ds=M_t, 
\end{align*}
which implies that the martingale is zero with probability one and so is $c_1^-(t)$. Thus the solution $c_1$ of \eqref{e.L112403} is non-negative.
\end{proof}

\begin{remark}
    If $c_1(t), c_2(t)$ solve system \eqref{e.L112403} with non-negative initial value, then for $i=1,2$ and $t\geq 0$, the average 
\begin{align*}
    \int_{\mathbb{T}^2}c_i(t, x) dx = \int_{\mathbb{T}^2}c_i(0) dx=:\bar{c}_i
\end{align*}
is conserved, which is equivalent to $\|c_i(t)\|_{L^1}=\|c_i(0)\|_{L^1}$ since the concentration is non-negative. Note that as $\rho=-\Delta \Phi$ which has zero spatial mean, consequently, the relation $\rho=c_1-c_2$ yields that $\bar{c}_1=\bar{c}_2$. 
\end{remark}
Let $U=(u,c_1,c_2)$ be a pathwise solution to the original NPNS system \eqref{e.L112403}, and denote by
\[\c_i = c_i - \bar{c}_i, \quad i=1,2.\]

\begin{proposition}\label{p.L122601}
Assume that $c_1(0),c_2(0)$ are non-negative almost surely. Then one has 
\begin{align}\label{eq.L122505}
    \|\c_1(t)\|_{L^2}^2+\|\c_2(t)\|_{L^2}^2\leq e^{-2Dt}(\|\c_1(0)\|_{L^2}^2 + \|\c_2(0)\|_{L^2}^2) \quad a.s..
\end{align}
\end{proposition}
\begin{proof}
Applying It\^o's formula to the equations of $c_1, c_2$ in \eqref{e.L112403} and using divergence free property of $u, b, \theta$,  we obtain
\begin{align*}
    \|\c_1(t)\|_{L^2}^2 &= \|\c_1(0)\|_{L^2}^2-2D\int_0^t\|\nabla \c_1(r)\|_{L^2}^2dr + 2D\int_0^t\langle \nabla\cdot (c_1\nabla\Phi), c_1 \rangle dr,\\
    \|\c_2(t)\|_{L^2}^2 &= \|\c_2(0)\|_{L^2}^2-2D\int_0^t\|\nabla \c_2(r)\|_{L^2}^2dr - 2D\int_0^t\langle \nabla\cdot (c_2\nabla\Phi), c_2 \rangle dr, 
\end{align*}  
where we also used the fact that the dual of $\mathcal{L}_{b_k}$ is $-\mathcal{L}_{b_k}$. 

As $c_1, c_2$ are non-negative, we have $c_1+c_2 \geq c_1-c_2=\rho$. Therefore, 
adding the two equations we have
\begin{align*}
    \langle \nabla\cdot (c_1\nabla\Phi), c_1 \rangle - \langle \nabla\cdot (c_2\nabla\Phi), c_2 \rangle = -\frac12\langle (c_1-c_2)^2,c_1+c_2\rangle\leq -\frac12\|\rho\|_{L^3}^3. 
\end{align*}
Consequently, 
\begin{align}\label{e.L122506}
\begin{split}
    \|\c_1(t)\|_{L^2}^2+\|\c_2(t)\|_{L^2}^2
    &\leq \|\c_1(0)\|_{L^2}^2 + \|\c_2(0)\|_{L^2}^2\\
    &-D\int_0^t\|\rho\|_{L^3}^3dr-2D\int_0^t(\|\nabla \c_1(r)\|_{L^2}^2+\|\nabla \c_2(r)\|_{L^2}^2)dr.
\end{split}
\end{align}
Recall that the constant of the Poincar\'e inequality is 1, we have
\begin{align}\label{e.L122507}
\|\c_1(t)\|_{L^2}^2+\|\c_2(t)\|_{L^2}^2\leq e^{-2Dt}(\|\c_1(0)\|_{L^2}^2 + \|\c_2(0)\|_{L^2}^2) \quad a.s..
\end{align}
\end{proof}
It is clear that the proof of Lemma \ref{l.L110201} and Proposition \ref{p.L122601} are also valid for the cut-off system \eqref{sys:modified}. Hence the global pathwise solution $U^{\eta}=(u^{\eta},c_1^{\eta},c_2^{\eta})$ obtained in Theorem \ref{t.L122201} also satisfies the estimate \eqref{eq.L122505}. This aspect allows us to derive an estimate on the $U^{\eta}$ and show the existence of global solutions of the original NPNS system \eqref{e.L112403}. 

\begin{theorem}
For any $\Fc_0$ measurable initial data $U_0\in L^2(\Omega,\mch)$, there is a unique global pathwise solution to the NPNS system \eqref{e.L112403}. 
\end{theorem}
\begin{proof}
For each $n\in\mathbb{N}$, let $U^n=(u^n,c_1^n,c_2^n)$ be the solution of the cut-off system \eqref{sys:modified} with $\eta=n$ according to Theorem \ref{t.L122201}. Define the stopping times
\begin{align*}
    \tau^n(\omega)= \begin{cases}\inf \left\{t>0:\left\|c_1^{n}(t)\right\|_{L^2}+\left\|c_2^{n}(t)\right\|_{L^2} \geq n / 2\right\}, & \text { if }\left\|c_1(0)\right\|_{L^2}+\left\|c_2(0)\right\|_{L^2} <n / 2, \\ 0, & \text { if }\left\|c_1(0)\right\|_{L^2}+\left\|c_2(0)\right\|_{L^2} \geq n / 2.\end{cases}
\end{align*}
By the definition of the stopping time $\tau^n$ and the cut-off function $\tau_n(\|c_1\|_{L^2} + \|c_2\|_{L^2})$, for $t\in[0,\tau^n]$ we have $\tau_n(\|c_1\|_{L^2} + \|c_2\|_{L^2})=1$. Consequently, $U^n$ is also the unique solution to the original NPNS system \eqref{e.L112403} on $t\in[0,\tau^n]$.
By uniqueness, the sequence $\tau^n$ is non-decreasing almost surely and $U^n=U^m$ on $[0,\tau^n\wedge\tau^m]$. 
Let $\tau=\sup_{n}\tau^n$ and define $U=U^{n}$ on $[0,\tau^n]$ for any $n\in \mathbb N$. As $U_0\in L^2(\Omega,\mathcal H)$, we have $\mathbb{P}(\tau>0)=1$. Moreover, $(U,\tau)$ is a local pathwise solution of the original NPNS system \eqref{e.L112403}. 
It is a global pathwise solution if we can show $\tau=\infty$ almost surely.
Denote $S=\{\omega\in\Omega: \tau<\infty\}.$ Then one has 
\[S=\bigcup_{N}\bigcap_{n}\{\tau^n\leq N\}.\]
It follows from the estimate \eqref{eq.L122505} for $U^{n}$ and the Markov inequality that 
\begin{align*}
   \mathbb{P}(\tau^n\leq N) &=\mathbb{P} \left(\sup_{t\in[0,N]}(\|c_1^n(t)\|_{L^2}+\|c_2^n(t)\|_{L^2})>\frac{n}2\right)\\
   &\leq \frac{2\E \sup\limits_{t\in[0,N]}(\|c_1^n(t)\|_{L^2}+\|c_2^n(t)\|_{L^2})}{n}\leq \frac{C}{n},
\end{align*}
where the constant $C>0$ depending on $\E c_1(0)$, $\E c_2(0)$, but not on $n$. Therefore, it follows that for every $N\in \mathbb N$,
\[\mathbb P\left(\bigcap_{n}\{\tau^n\leq N\}\right)=\lim_{n\to\infty} \mathbb{P}(\tau^n\leq N)=0,\]
which implies $\mathbb P(S)=0$ and hence the global existence of the solution to  \eqref{e.L112403}.
\end{proof}

\section{Enhanced dissipation}\label{enhance}
Throughout this section, we denote by 
\[\c_i = c_i - \bar{c}_i, \quad i=1,2,\]
which have zero means. Let $U=(u,\c_1,\c_2)$ with norm 
\[\|U\|_{L^2}^2=\|u\|_{L^2}^2+\|\c_1\|_{L^2}^2+\|\c_2\|_{L^2}^2.\]
Under the special type of noise as described in \eqref{e.L122101}, the It\^o form of the system \eqref{e.w04091} becomes 
\begin{align}\label{e.L112401}
    \begin{split}
    dc_1&=\left((D+\kappa)\Delta c_1 -u\cdot \nabla c_1 +D  \nabla \cdot (c_1 \nabla \Phi) \right)dt + \sqrt{2\kappa}\sum_{k}\zeta_k\sigma_k\cdot\nabla c_1 dW_t^k, 
    \\
    dc_2&=\left((D+\kappa)\Delta c_2 -u\cdot \nabla c_2 -D  \nabla \cdot (c_2 \nabla \Phi)\right)dt + \sqrt{2\kappa}\sum_{k}\zeta_k\sigma_k\cdot\nabla c_2 dW_t^k,
    \\
    du&=\left(\nu\Delta u + S_{\zeta}(u)- \Pi\left( u\cdot \nabla u +\rho \nabla \Phi\right)\right)dt + \sqrt{2\kappa}\sum_{k}\zeta_k\Pi(\sigma_k\cdot\nabla u)dW_t^k, 
    \\
    \nabla\cdot u &=0,
    \\
    -\Delta \Phi &= \rho = c_1 - c_2.
    \end{split}
\end{align}
where 
\begin{align}\label{e.L112503}
    S_{\zeta}(u) = \kappa\sum_{k}\zeta_k^2\Pi\big[\sigma_k\cdot\nabla\Pi(\sigma_{-k}\cdot\nabla u)\big]. 
\end{align}
For the derivation, we refer the readers to \cite{luo2023enhanced}. 

\subsection{Energy inequality and exponential stability}
\begin{proposition}\label{p.L110901}
Let $\gamma_0$ be the constant from the Sobolev inequality
$\|f\|_{L^{\infty}}\leq \gamma_0\|f\|_{H^2},$
and denote by
$\bar{\c}_0 = \|\c_1(0)\|_{L^2}^2 + \|\c_2(0)\|_{L^2}^2.$
Assume 
\begin{align}\label{e.L112509}
    \nu D-2\gamma_0^2\bar{\c}_0>0,
\end{align}
and let 
$\gamma = \min\left\{\nu,2D-\frac{4\gamma_0^2\bar{\c}_0}{\nu}\right\}.$
Then for $\mathbb{P}$ almost surely one has
\begin{align}\label{e.L112501}
    \|U(t)\|_{L^2}^2+\gamma\int_s^t\|\nabla U(r)\|_{L^2}^2dr\leq \|U(s)\|_{L^2}^2, \quad t>s\geq0, 
\end{align}
and
\begin{align}\label{e.L112502}
    \|U(t)\|_{L^2}^2\leq e^{-\gamma t}\|U(0)\|_{L^2}^2, \quad t\geq 0. 
\end{align}
\end{proposition}
\begin{proof}
Recall that in the proof of Proposition \ref{p.L122601} we obtained 
\begin{align}\label{e.L111001}
\begin{split}
    \|\c_1(t)\|_{L^2}^2+\|\c_2(t)\|_{L^2}^2
    &\leq \|\c_1(s)\|_{L^2}^2 + \|\c_2(s)\|_{L^2}^2\\
    &-D\int_s^t\|\rho\|_{L^3}^3dr-2D\int_s^t(\|\nabla \c_1(r)\|_{L^2}^2+\|\nabla \c_2(r)\|_{L^2}^2)dr.
\end{split}
\end{align}
and 
\begin{align}\label{e.L111004}
    \|\c_1(t)\|_{L^2}^2+\|\c_2(t)\|_{L^2}^2\leq e^{-2Dt}(\|\c_1(0)\|_{L^2}^2 + \|\c_2(0)\|_{L^2}^2).
\end{align}
Applying It\^o's formula to the equation of $u$ in \eqref{e.L112403} we have 
\begin{align}\label{e.L111002}
    \|u(t)\|_{L^2}^2 = \|u(s)\|_{L^2}^2 - 2\nu\int_s^t\|\nabla u(r)\|_{L^2}^2dr - 2\int_{s}^t\langle \rho\nabla\Phi, u\rangle dr.
\end{align}
By the elliptic estimate, the H\"older inequality, Young's inequality, and the Poincar\'e inequality,
\begin{align*}
    \langle \rho\nabla\Phi, u\rangle \leq \|\rho\|_{L^2}\|\nabla\Phi\|_{L^{\infty}}\|u\|_{L^2}
    &\leq \gamma_0\|\rho\|_{L^2}\|\nabla\rho\|_{L^2}\|u\|_{L^2}\\
    &\leq \frac{\gamma_0^2}{2\nu}\|\rho\|_{L^2}^2\|\nabla\rho\|_{L^2}^2+ \frac{\nu}{2}\|\nabla u\|_{L^2}^2
    \\
    &\leq \frac{\gamma_0^2}{2\nu} \|\rho-\bar{\rho}\|_{L^2}^2\|\nabla\rho\|_{L^2}^2+ \frac{\nu}{2}\|\nabla u\|_{L^2}^2
    \\
    & \leq \frac{2\gamma_0^2}{\nu}(\| \c_1\|_{L^2}^2+\|\c_2\|_{L^2}^2)(\|\nabla \c_1\|_{L^2}^2+\|\nabla \c_2\|_{L^2}^2) + \frac{\nu}{2}\|\nabla u\|_{L^2}^2
    \\
    &\leq \frac{2\gamma_0^2\bar{\c}_0}{\nu}(\|\nabla \c_1\|_{L^2}^2+\|\nabla \c_2\|_{L^2}^2)+ \frac{\nu}{2}\|\nabla u\|_{L^2}^2,
\end{align*}
where in the last step we used 
\eqref{e.L111004} with the notation
$\bar{\c}_0 = \|\c_1(0)\|_{L^2}^2 + \|\c_2(0)\|_{L^2}^2.$
Combining this estimate with \eqref{e.L111001} and \eqref{e.L111002}, use the Poincar\'e inequality, we obtain 
\begin{align}\label{e.L112506}
    \|U(t)\|_{L^2}^2\leq \|U(s)\|_{L^2}^2-\nu\int_{s}^{t}\|\nabla u(r)\|_{L^2}^2dr - \left(2D-\frac{4\gamma_0^2\bar{\c}_0}{\nu}\right)\int_{s}^{t}(\|\c_1(r)\|_{L^2}^2+\|\c_2(r)\|_{L^2}^2)dr.
\end{align}
This proves the desired energy inequality \eqref{e.L112501}, which in turn implies \eqref{e.L112502} by Gronwall's inequality and the Poincar\'e inequality. 
\end{proof}
\begin{proposition}\label{p.L111001}
For any initial data $c_1(0), c_2(0), u(0)$, there is a (random) constant $C>0$ depending on $\|u(0)\|_{L^2}^2, \bar{\c}_0$ such that 
\begin{align*}
     \|u(t)\|_{L^2}^2+\|\c_1(t)\|_{L^2}^2+\|\c_2(t)\|_{L^2}^2
    \leq C e^{-\Gamma t}, \quad t\geq 0,    
\end{align*}
where $\Gamma$ is a constant depending on $\nu$ and $D$. 
\end{proposition}
\begin{proof}
In equation \eqref{e.L111002}, we can use elliptic estimate to obtain that 
\begin{align*}
    \langle \rho\nabla\Phi, u\rangle \leq C\|\rho\|_{L^2}\|\nabla\Phi\|_{\infty}\|u\|_{L^2}\leq C\|\rho\|_{L^2}\|\rho\|_{L^3}\|u\|_{L^2} = C\|\rho-\bar{\rho}\|_{L^2}\|\rho\|_{L^3}\|u\|_{L^2}.
\end{align*}
Thanks to inequality \eqref{e.L111004} and the Ponicar\'e inequality, one has the integral inequality
\begin{align*}%\label{e.L111003}
    \|u(t)\|_{L^2}^2 &\leq \|u(0)\|_{L^2}^2 - 2\nu\int_0^t\| u(r)\|_{L^2}^2dr +C(\bar{\c}_0)\int_{0}^te^{-2Dr}\|\rho\|_{L^3}\|u\|_{L^2} dr\\
    &\leq \|u(0)\|_{L^2}^2 - \nu\int_0^t\| u(r)\|_{L^2}^2dr +C(\bar{\c}_0, \nu)\int_{0}^te^{-12Dr} dr+ \frac13\int_{0}^t\|\rho\|_{L^3}^3dr,
\end{align*}
where we used Young's inequality for products. Inequality \eqref{e.L111001} implies that $\int_{0}^t\|\rho\|_{L^3}^3dr$ is bounded by $\bar{\c}_0$ for $t\geq 0$. Thus we have 
\begin{align*}
    \|u(t)\|_{L^2}^2 \leq C - \nu\int_0^t\| u(r)\|_{L^2}^2dr,  
\end{align*}
where the constant $C$ depends on $\bar{\c}_0,\|u(0)\|_{L^2}^2, \nu, D$. Applying Gronwall's inequality and combining with \eqref{e.L111004} we obtain the desired estimate. 
\end{proof}

\subsection{Enhanced dissipation}
Now we take the special form of noise as in \eqref{e.L112401} and prove that appropriate choice of the parameters gives the desired enhanced dissipation. Since the mean $\bar{c}_1=\bar{c}_2$ is space time independent, we can rewrite equations \eqref{e.L112401} in terms of $(u, \c_1,\c_2)$ as 
\begin{align}\label{e.L112404}
    \begin{split}
    d\c_1&=\left((D+\kappa)\Delta \c_1 -u\cdot \nabla \c_1 +D  \nabla \cdot (\c_1 \nabla \Phi) -D\bar{c}_1\rho\right)dt + \sqrt{2\kappa}\sum_{k}\zeta_k\sigma_k\cdot\nabla \c_1 dW_t^k, 
    \\
    d\c_2&=\left((D+\kappa)\Delta \c_2 -u\cdot \nabla \c_2 -D  \nabla \cdot (\c_2 \nabla \Phi) +D\bar{c}_2\rho\right)dt + \sqrt{2\kappa}\sum_{k}\zeta_k\sigma_k\cdot\nabla \c_2 dW_t^k,
    \\
    du&=\left(\left(\nu+\frac{\kappa}{4}\right)\Delta u + \left(S_{\zeta}(u) -\frac{\kappa}{4}\Delta u\right) - \Pi\left( u\cdot \nabla u +\rho \nabla \Phi\right)\right)dt + \sqrt{2\kappa}\sum_{k}\zeta_k\Pi(\sigma_k\cdot\nabla u)dW_t^k, 
    \\
    \nabla\cdot u &=0,
    \\
    -\Delta \Phi &= \rho = \c_1 - \c_2.
    \end{split}
\end{align}
where 
\begin{align*}
    S_{\zeta}(u) = \kappa\sum_{k}\zeta_k^2\Pi\big[\sigma_k\cdot\nabla\Pi(\sigma_{-k}\cdot\nabla u)\big]. 
\end{align*}
Denote by $Q_t = e^{(D+\kappa)t\Delta}$ and $P_t = e^{(\nu+\kappa/4)t\Delta}$, then the above equations for $(u,\c_1, \c_2)$ have the following mild form 
\begin{align}\label{e.L112405}
    \begin{split}
    \c_1(t)&=Q_{t-s}\c_1(s)-\int_s^tQ_{t-r}\left(u\cdot \nabla \c_1-D  \nabla \cdot (\c_1 \nabla \Phi) +D\bar{c}_1\rho\right)(r)dr + Z_{s,t}(\c_1),
    \\
    \c_2(t)&=Q_{t-s}\c_2(s)-\int_s^tQ_{t-r}\left(u\cdot \nabla \c_2+D  \nabla \cdot (\c_2 \nabla \Phi) -D\bar{c}_2\rho\right)(r)dr + Z_{s,t}(\c_2),
    \\
    u(t)&=P_{t-s}u(s)-\int_s^t P_{t-r} \Pi\left( u\cdot \nabla u +\rho \nabla \Phi\right)(r)dr + \int_s^t P_{t-r}\left(S_{\zeta}(u) -\frac{\kappa}{4}\Delta u\right)(r)dr + Z_{s,t}(u),
    \\
\end{split}
\end{align}
where 
\begin{align*}
    Z_{s, t}(\c_i) &= \sqrt{2\kappa}\sum_{k}\zeta_k\int_s^tQ_{t-r}\sigma_k\cdot\nabla \c_i(r)dW_r^k, \quad i=1,2,\\
    Z_{s, t}(u) &= \sqrt{2\kappa}\sum_{k}\zeta_k\int_s^tP_{t-r}\Pi(\sigma_k\cdot\nabla u(r))dW_r^k. 
\end{align*}
We have the following estimates on the stochastic convolutions. 
\begin{lemma}\label{l.L112502}
Let $n\geq 0$ and $0<\alpha<1<\beta\leq 3$. For any initial data $c_1(0), c_2(0)\in L^2(\Omega, L^2)$ and $u(0)\in L^2(\Omega, H)$, the stochastic convolutions $Z_{s,t}(\c_i), i=1,2$ for the corresponding solution satisfies
\begin{align}\label{e.L112513}
    &\int_n^{n+1}\mathbb{E}(\|Z_{n,t}(\c_1)\|_{L^2}^2+\|Z_{n,t}(\c_2)\|_{L^2}^2)dt
    \\
    &\hspace{1cm}\lesssim_{\alpha,\beta} \kappa D^{-\frac{\beta}{\alpha+\beta}}(D+\kappa)^{-\frac{\alpha(\beta+3)}{2(\alpha+\beta)}}\|\zeta\|_{\ell^{\infty}}^{\frac{2\alpha}{\alpha+\beta}}\mathbb{E}(\|\c_1(n)\|_{L^2}^2+\|\c_2(n)\|_{L^2}^2). 
\end{align}
If in addition the condition \eqref{e.L112509} holds, then we have 
\begin{align}\label{e.L112510}
     \int_n^{n+1}\mathbb{E}\|Z_{n,t}(u)\|_{L^2}^2dt\lesssim_{\alpha,\beta} \kappa \nu^{-\frac{\beta}{\alpha+\beta}}(\nu+\kappa/4)^{-\frac{\alpha(\beta+3)}{2(\alpha+\beta)}}\|\zeta\|_{\ell^{\infty}}^{\frac{2\alpha}{\alpha+\beta}}\mathbb{E}\|U(n)\|_{L^2}^2.
\end{align}
\end{lemma}
\begin{proof}
    The above estimate \eqref{e.L112513} on $Z_{s,t}(\c_i), i=1,2$ follows immediately from Theorem 3.4 in \cite{luo2023enhanced} and the energy inequality \eqref{e.L111001}. The estimate on $Z_{s,t}(u)$ follows from the energy inequality \eqref{e.L112506} and the same proof of Theorem 3.4 in \cite{luo2023enhanced}. Indeed, one only needs to replace the energy inequality in \cite{luo2023enhanced} with the following
 \[\|u(t)\|_{L^2}^2\leq \|U(s)\|_{L^2}^2-\nu\int_{s}^{t}\|\nabla u(r)\|_{L^2}^2dr, \quad \forall t>s\geq 0,\]
 which follows from \eqref{e.L112506}. 
\end{proof}

\begin{lemma} \label{l.w111501}
Let $L>0$. Assume that the initial data $(u(0),c_1(0),c_2(0))\in L^2(\Omega, H\times L^2\times L^2)$ satisfies $\mathbb{P}$ almost surely that
\begin{align}\label{e.L112602}
\|u(0)\|_{L^2}+\|c_1(0)\|_{L^2}+\|c_2(0)\|_{L^2}\leq L,   
\end{align}
and  the condition \eqref{e.L112509}, i.e, 
$\nu D-2\gamma_0^2\bar{\c}_0>0.$
Then there exists $\delta>0$ such that
\begin{equation*}
\mathbb{E}\|U(n+1)\|_{L^2}^2\leq \delta\mathbb{E}\|U(n)\|_{L^2}^2 \quad \forall n\in\mathbb N,
\end{equation*}
where, for some $0<\alpha<1<\beta \leq 3$,
\begin{align*}
   \delta \lesssim_{\nu,D,L,\alpha,\beta} \frac{1}{\kappa}+\frac{1}{\kappa^2}+\frac{1}{N^2} + \kappa^{\frac{2\beta-\alpha(\beta+1)}{2(\alpha+\beta)}}N^{-\frac{2\alpha}{\alpha+\beta}}.
\end{align*}

\end{lemma}
\begin{proof}[Proof of Lemma \ref{l.w111501}]
It follows from the energy inequality \eqref{e.L112501} that $\|U(t)\|_{L^2}^2$ is non-increasing in time. Thus for $n\in\mathbb{N}$, 
\begin{align}\label{e.L112603}
    \|U(n+1)\|_{L^2}^2\leq \int_n^{n+1}\|U(t)\|_{L^2}^2dt. 
\end{align}
We first deal with estimates on $\c_i, i=1, 2$, by using the mild formulation \eqref{e.L112405} with $s=n$. We have 
\begin{align*}
    I_1:=\int_n^{n+1}\|Q_{t-n}\c_i(n)\|_{L^2}^2dt\leq \int_n^{n+1} e^{-t(D+\kappa)}\|\c_i(n)\|_{L^2}^2dt\lesssim\frac{1}{\kappa}\|U(n)\|_{L^2}^2. 
\end{align*}
Thanks to Lemma \ref{l.L112501}, one has 
\begin{align*}
    I_2: = \int_n^{n+1}\left\|\int_n^tQ_{t-r}u(r)\cdot\nabla \c_i(r)dr\right\|_{L^2}^2dt\lesssim \frac{1}{\kappa^2}\int_n^{n+1}\|u(r)\cdot\nabla \c_i(r)\|_{H^{-2}}^2dr. 
\end{align*}
For any $\phi\in H^2$, by the Sobolev embedding $H^2\subset L^{\infty}$, we have 
\begin{align}\label{e.L112505}
    |\langle u\cdot\nabla \c_i, \phi\rangle| \lesssim \|u\|_{L^2}\|\c_i\|_{H^1}\|\phi\|_{L^{\infty}}\lesssim \|u\|_{L^2}\|\c_i\|_{H^1}\|\phi\|_{H^2}.
\end{align}
Therefore, by estimates \eqref{e.L112502} and \eqref{e.L111001}, we have 
\begin{align*}
    I_2\lesssim \frac{1}{\kappa^2}\int_n^{n+1}\|u(r)\|_{L^2}^2\|\c_i(r)\|_{H^1}^2dr\lesssim \frac{1}{\kappa^2}\|U(0)\|_{L^2}^2\|U(n)\|_{L^2}^2. 
\end{align*}
In a similar fashion, we estimate 
\begin{align*}
    I_3: = D\int_n^{n+1}\left\|\int_n^tQ_{t-r}\nabla\cdot(\c_i(r)\nabla\Phi(r))dr\right\|_{L^2}^2dt\lesssim \frac{1}{\kappa^2}\int_n^{n+1}\left\|\nabla\cdot(\c_i(r)\nabla\Phi(r))\right\|_{H^{-2}}^2dr. 
\end{align*}
For any $\phi\in H^2$,
\begin{align*}
    |\langle\nabla\cdot(\c_i\nabla\Phi), \phi\rangle|=|\langle\c_i\nabla\Phi, \nabla\phi\rangle|\lesssim \|\c_i\|_{L^2}\|\nabla\Phi\|_{L^{\infty}}\|\nabla\phi\|_{L^2}&\lesssim \|\c_i\|_{L^2}\|\nabla\rho\|_{L^2}\|\phi\|_{H^1}\\
    &\lesssim \|\c_i\|_{L^2}(\|\nabla\c_1\|_{L^2}+\|\nabla\c_2\|_{L^2})\|\phi\|_{H^1}.
\end{align*}
Hence by estimates \eqref{e.L111001} and \eqref{e.L111004} it follows that 
\begin{align}\label{e.L112507}
    I_3\lesssim \frac{1}{\kappa^2}\|U(0)\|_{L^2}^2\int_n^{n+1}(\|\nabla\c_1\|_{L^2}^2+\|\nabla\c_2\|_{L^2}^2) dr\lesssim \frac{1}{\kappa^2}\|U(0)\|_{L^2}^2\|U(n)\|_{L^2}^2. 
\end{align}
Next, we estimate 
\begin{align*}
    I_4: = D\bar{c}_i\int_n^{n+1}\left\|\int_n^tQ_{t-r}\rho(r)dr\right\|_{L^2}^2dt\lesssim \frac{1}{\kappa^2}\bar{c}_i\int_{n}^{n+1}\|\rho(r)\|_{H^{-2}}^2dr \lesssim \frac{1}{\kappa^2}\bar{c}_i\|U(n)\|_{L^2}^2.
\end{align*}
Combining estimates on $I_1$ to $I_4$, we obtain 
\begin{align}\label{e.L112601}
\begin{split}
    &\int_n^{n+1}(\|\c_1(t)\|_{L^2}^2+\|\c_2(t)\|_{L^2}^2)dt\\
    &\qquad\lesssim \left(\frac{1}{\kappa}+\frac{1}{\kappa^2}(\|U(0)\|^2+\bar{c}_1+\bar{c}_2)\right)\|U(n)\|_{L^2}^2 + \int_n^{n+1}\|Z_{n,t}(\c_1)\|_{L^2}^2dt+\int_n^{n+1}\|Z_{n,t}(\c_2)\|_{L^2}^2dt.
\end{split} 
\end{align}

We now estimate $\int_n^{n+1}\|u(t)\|_{L^2}^2dt$. We have 
\begin{align*}
    J_1:= \int_n^{n+1}\|P_{t-n}u(n)\|_{L^2}^2 dt\leq \int_n^{n+1}e^{-t(\nu+4/\kappa)}\|u(n)\|_{L^2}^2dt\lesssim \frac{1}{\kappa}\|U(n)\|_{L^2}^2. 
\end{align*}
By Lemma \ref{l.L112501}, the boundedness of the Leray projection $\Pi$, and the Sobolev embedding $H^2\subset L^{\infty}$ argument as in \eqref{e.L112505}, it follows that 
\begin{align*}
    J_2: = \int_n^{n+1}\left\|\int_n^tP_{t-r}\Pi(u(r)\cdot\nabla u(r))dr\right\|_{L^2}^2dt
    &\lesssim\frac{1}{\kappa^2} \int_n^{n+1}\|u(r)\cdot\nabla u(r)\|_{H^{-2}}^2dr\\
    &\lesssim \frac{1}{\kappa^2} \int_n^{n+1}\|u(r)\|_{L^2}^2\|u(r)\|_{H^1}^2dr\\
    &\lesssim \frac{1}{\kappa^2}\|U(0)\|_{L^2}^2\|U(n)\|_{L^2}^2,
\end{align*}
where in the last step we used the estimates \eqref{e.L112502} and \eqref{e.L112506}. 

Similar to the estimate of $I_3$ in \eqref{e.L112507} we obtain 
\begin{align*}
    J_3: = \int_n^{n+1}\left\|\int_n^tP_{t-r}\Pi(\rho(r)\nabla\Phi(r))dr\right\|_{L^2}^2dt
    &\lesssim\frac{1}{\kappa^2}\int_n^{n+1}\|\rho(r)\nabla\Phi(r)\|_{H^{-2}}^2dr\\
    &\leq \frac{1}{\kappa^2}\int_n^{n+1}\|\rho(r)\|_{L^2}^2\|\nabla\rho(r)\|_{L^2}^2dr\lesssim \frac{1}{\kappa^2}\|U(0)\|_{L^2}^2\|U(n)\|_{L^2}^2. 
\end{align*}

Applying Lemma \ref{l.L112501}, Lemma \ref{corrector} and estimate \eqref{e.L112506} one obtains 
\begin{align*}
    J_4: &= \int_n^{n+1}\left\|\int_n^t P_{t-r}\left(S_{\zeta}(u) -\frac{\kappa}{4}\Delta u\right)(r)dr\right\|_{L^2}^2dt\\
    &\lesssim \int_n^{n+1}\left\|S_{\zeta}(u(r)) -\frac{\kappa}{4}\Delta u(r)\right\|_{H^{-2}}^2dr\lesssim \frac{1}{\kappa^2}\int_n^{n+1}\frac{\kappa^2}{N^2}\|u(r)\|_{H^1}dr\leq \frac{1}{N^2\nu}\|U(n)\|_{L^2}^2. 
\end{align*}
Combining estimates on $J_1$ to $J_4$, we obtain 
\begin{align}\label{e.L112508}
    \int_n^{n+1}\|u(t)\|_{L^2}^2dt\lesssim \left(\frac{1}{\kappa}+ \frac{1}{\kappa^2}\|U(0)\|_{L^2}^2+\frac{1}{N^2}\right)\|U(n)\|_{L^2}^2 + \int_n^{n+1}\|Z_{n,t}(u)\|_{L^2}^2dt. 
\end{align}
Thanks to Lemma \ref{l.L112502}, inequalities \eqref{e.L112603}, \eqref{e.L112601}, and \eqref{e.L112508}, and the uniform boundedness assumption \eqref{e.L112602}, we have 
\begin{align*}
    \mathbb{E}\|U(n+1)\|_{L^2}^2\leq \delta\mathbb{E}\|U(n)\|_{L^2}^2
\end{align*}
with 
\begin{align*}
   \delta \lesssim_{\nu,D,L,\alpha,\beta} \frac{1}{\kappa}+\frac{1}{\kappa^2}+\frac{1}{N^2} + \kappa^{\frac{2\beta-\alpha(\beta+1)}{2(\alpha+\beta)}}N^{-\frac{2\alpha}{\alpha+\beta}},
\end{align*}
where we used the fact that $\|\zeta\|_{\ell^{\infty}}\sim 1/N$, see \eqref{e.L112504} and \cite{luo2023enhanced}. 
\end{proof}

Now we are ready to prove Theorem \ref{t.w04093} on the enhanced dissipation.
\begin{proof} [Proof of Theorem \ref{t.w04093}]
First note that Lemma \ref{l.w111501} with deterministic initial data implies the existence  a $\delta>0$ such that
\begin{equation}\label{e.w11271}
    \begin{split}
\mathbb{E}\|U(n)\|_{L^2}^2
\leq
\delta\mathbb{E}\|U(n-1)\|_{L^2}^2
\leq
\delta^2\mathbb{E}\|U(n-2)\|_{L^2}^2
\leq
\dots
\leq
\delta^n\|U(0)\|_{L^2}^2.
    \end{split}
\end{equation}

By \eqref{e.L112501}, we know that $\|U(t)\|_{L^2}^2$ is decreasing in time $t$ a.s., thus for $t\in [n, n+1]$, we have
\begin{equation}\label{e.w11272}
    \begin{split}
\mathbb{E}\left(\sup_{t\in [n, n+1]}\|U(t)\|_{L^2}^2\right)
=
\mathbb{E}\|U(n)\|_{L^2}^2
\leq
\delta^n\|U(0)\|_{L^2}^2
=
e^{-2\lambda' n}\|U(0)\|_{L^2}^2
,
    \end{split}
\end{equation}
where $\lambda'=-\frac{1}{2}\log \delta$. Given any $\lambda>0$, by Lemma \ref{l.w111501} we can choose $\delta$ such that $\lambda'>\lambda$. 
Then for random initial data satisfying the conditions of Lemma \ref{l.w111501}, we have by Markov property that 
\begin{align*}
    \mathbb{E}\left(\sup_{t\in [n, n+1]}\|U(t)\|_{L^2}^2\mid \mathcal{F}_0\right)\leq e^{-2\lambda' n}\|U(0)\|_{L^2}^2 \quad a.s..
\end{align*}

Next, for $n\geq 1$, we define the following set
\begin{equation}\label{e.w11273}
    \begin{split}
    A_n
    =\left\{\omega\in \Omega: \sup_{t\in [n, n+1]}\|U(t,\omega)\|_{L^2}^2> e^{-\lambda n}\|U(0)\|_{L^2}^2 \right\}.
    \end{split}
\end{equation}
By the conditional Markov inequality, we have 
\begin{align}\label{e.L112701}
    \mathbb{P}(A_n\mid \mathcal{F}_0)\leq \frac{e^{2\lambda n}}{\|U(0)\|_{L^2}^2}\mathbb{E}\left(\sup_{t\in [n, n+1]}\|U(t)\|_{L^2}^2\mid \mathcal{F}_0\right)\leq e^{2(\lambda - \lambda')n} \quad a.s..
\end{align}
Taking expectation on both sides of inequality \eqref{e.L112701} we obtain 
\begin{align*}
    \sum_{n}\mathbb P(A_n)\leq \sum_{n} e^{2(\lambda - \lambda')n}<\infty.
\end{align*}

By Borel-Cantelli lemma, we know that for $\omega \in \Omega$, there exists large $N(\omega)\geq 1$, such that for $n > N(\omega)$
\begin{equation}\label{e.w11275}
    \begin{split}
    \sup_{t\in [n, n+1]}\|U(t,\omega)\|_{L^2}^2\leq e^{-\lambda n}\|U(0)\|_{L^2}^2.
    \end{split}
\end{equation}
On the other hand, for $0\leq n\leq N(\omega)$, we use the property of decreasing in time and get
\begin{equation}\label{e.w11281}
    \begin{split}
    \sup_{t\in [n, n+1]}\|U(t,\omega)\|_{L^2}
    \leq \|U(n,\omega)\|_{L^2}
    &=
    e^{\lambda n}e^{-\lambda n}
    \|U(n,\omega)\|_{L^2}
    \\&\leq
    e^{\lambda N(\omega)}e^{-\lambda n}
    \|U(0)\|_{L^2}.
    \end{split}
\end{equation}
Thus by taking $C(\omega)=e^{\lambda(1+N(\omega))}$, we obtain 
\begin{align*}
    \|U(t)\|_{L^2}\leq C(\omega)e^{-\lambda t}\|U(0)\|_{L^2}. 
\end{align*}

Lastly, we should verify that $C(\omega)$ has finite $p$-th moment, for which we need to
estimate the tail probability of $\mathbb P(\{N(\omega)\geq k\})$. Without loss of generality, we take $N(\omega)$ to be the largest integer $n$ such that 
\begin{equation}\label{e.w12041}
    \begin{split}
\sup_{t\in [n, n+1]}\|U(t,\omega)\|_{L^2}^2> e^{-\lambda n}\|U(0)\|_{L^2}^2.
\end{split}
\end{equation}
Recall \eqref{e.w11273}, we have
\begin{equation}\label{e.w12042}
    \begin{split}
\{\omega \in \Omega: N(\omega)\geq k\}
=
\bigcup_{n=k}^{\infty} A_n.
\end{split}
\end{equation}
This yields that
\begin{equation}\label{e.w12043}
\begin{split}
\mathbb P(\{N(\omega)\geq k\})
\leq
\sum_{n=k}^{\infty} \mathbb P (A_n)
\leq
\sum_{n=k}^{\infty} e^{2(\lambda-\lambda')n}
=
\frac{e^{2(\lambda-\lambda')k}}{1-e^{2(\lambda -\lambda')}}.
\end{split}
\end{equation}
Finally, we see that
\begin{equation}\label{e.w12044}
\begin{split}
\mathbb Ee^{\lambda p N(\omega)}
=
\sum_{k=0}^{\infty}e^{\lambda p k} \mathbb P(N(\omega)= k)
\leq
\frac{1}{1-e^{2(\lambda -\lambda')}}
\sum_{k=0}^{\infty} e^{\lambda pk}e^{2(\lambda-\lambda')k}<\infty,
\end{split}
\end{equation}
by further choosing $\delta$ to make $\lambda p+2(\lambda-\lambda')<0$ according to Lemma \ref{l.w111501}. This implies the constant $C(\omega)$ has finite $p$-th moment.

\end{proof}

\section*{Acknowledgements}	
QL was partially supported by an AMS-Simons travel grant. WW was partially supported by the NSF grant DMS-1928930 while as a member of the program ``Mathematical Problems in Fluid Dynamics'' at the Simons Laufer Mathematical Sciences Institute (formerly MSRI) in Berkeley, California, during the summer of 2023, an AMS-Simons travel grant, and a postdoctoral collaborative research grant through the University of Arizona Department of Mathematics.

\section*{Appendix: heat kernel and corrector estimates}

Here we recall several known facts and results from \cite{luo2023enhanced}.

\begin{lemma}\label{l.L112501}
    For any $a<b, \alpha \in \mathbb{R}$ and any $f \in L^{2}\left(a, b ; H^{\alpha}\right)$, it holds 

\begin{equation}
\left\|\int_{a}^{t} e^{\delta(t-s) \Delta} f_{s} \mathrm{~d} s\right\|_{H^{\alpha+1}}^{2} \lesssim \frac{1}{\delta} \int_{a}^{t}\left\|f_{s}\right\|_{H^{\alpha}}^{2} \mathrm{~d} s, \quad \forall t \in[a, b].
\end{equation}

Similarly, we have

\begin{equation}
\int_{a}^{b}\left\|\int_{a}^{t} e^{\delta(t-s) \Delta} f_{s} \mathrm{~d} s\right\|_{H^{\alpha+2}}^{2} \mathrm{~d} t \lesssim \frac{1}{\delta^{2}} \int_{a}^{b}\left\|f_{s}\right\|_{H^{\alpha}}^{2} \mathrm{~d}s.
\end{equation}

\end{lemma}

\begin{theorem}\label{corrector}
Let $S_{\zeta}$ be defined as in \eqref{e.L112503} with $\zeta_k$ given by \eqref{e.L112504} and $N \geq 1$. Then there exists a constant $C>0$, independent of $N \geq 1$, such that for any $s \in \mathbb{R}$, $\alpha \in[0,1]$, and any divergence free field $v \in H^{s}\left(\mathbb{T}^{2}, \mathbb{R}^{2}\right)$, one has 
\begin{equation}\label{e.w111501}
\left\|S_{\zeta}(v)-\frac{1}{4} \kappa \Delta v\right\|_{H^{s-2-\alpha}} \leq C \frac{\kappa}{N^{\alpha}}\|v\|_{H^{s}}.
\end{equation}
\end{theorem}

\bibliographystyle{abbrv}
\bibliography{npns}

\end{document}